[1994/12/01]
\documentclass{amsart}
\chardef\bslash=`\\ 

\usepackage{amssymb}
\usepackage{natbib}
\bibpunct{[}{]}{,}{a}{}{;} 

\usepackage{color}
\usepackage{graphicx}

\usepackage{caption}
\usepackage{wrapfig}
\usepackage{lscape}
\usepackage{rotating}

\usepackage{overpic}

\usepackage{adjustbox}

\usepackage{hvfloat,lipsum}

\usepackage{caption}
\usepackage{subcaption}

\usepackage{appendix}

\usepackage{url}

\usepackage{accents}

\newtheorem{thm}{Theorem}[section]

\newtheorem{lem}[thm]{Lemma}

\newtheorem{conjecture}{Conjecture}
\newtheorem*{conjecture*}{Conjecture}

\theoremstyle{definition}

\newtheorem{rem}{Remark}
\newtheorem*{rem2}{Remark}

\theoremstyle{remark}

\newcommand{\conref}[1]{Conjecture~\ref{#1}}
\newcommand{\secref}[1]{Section~\ref{#1}}
\newcommand{\lemref}[1]{Lemma~\ref{#1}}

\newcommand{\figref}[1]{Figure~\ref{#1}}

\newcommand{\E}{\mathrm{E}}

\newcommand{\Var}{\mathrm{Var}}

\definecolor{gray}{rgb}{0.5,0.5,0.5}

\newcommand{\obsgray}[1]{\textcolor{gray}{#1}}

\newcommand{\divides}{\mid}

\newcommand{\e}{\mathrm{e}}

\newcommand{\B}{\mathcal{B}}

\DeclareMathOperator{\li}{li}

\makeatletter
\def\imod#1{\allowbreak\mkern10mu({\operator@font mod}\,\,#1)}
\makeatother

\newcommand{\eval}[2][\right]{\relax
  \ifx#1\right\relax \left.\fi#2#1\rvert}

\title[Primes in the intervals between primes squared]{Primes in the intervals between primes squared}
\author{Kolbj{\o}rn Tunstr{\o}m}
\email{kolbjorn@chalmers.se}
\address{Complex Systems Group, Department of Energy and Environment, Chalmers University of Technology, 41296 Gothenburg, Sweden}

\begin{document}

\begin{abstract} 
The set of short intervals between consecutive primes squared has the pleasant---but seemingly unexploited---property that each interval $s_k:=\{p_k^2, \dots,p_{k+1}^2-1\}$ is fully sieved by the $k$ first primes. Here we take advantage of this essential characteristic and present evidence for the conjecture that $\pi_k \sim |s_k|/ \log p_{k+1}^2$, where $\pi_k$ is the number of primes in $s_k$; or even stricter, that $y=x^{1/2}$ is both necessary and sufficient for the prime number theorem to be valid in intervals of length $y$. In addition, we propose and substantiate that the prime counting function $\pi(x)$ is best understood as a sum of correlated random variables $\pi_k$. Under this assumption, we derive the theoretical variance of $\pi(p_{k+1}^2)=\sum_{j=1}^k  \pi_j$, from which we are led to conjecture that $|\pi({x})-\li(x)| =O(\sqrt{\li(x)})$. Emerging from our investigations is the view that the intervals between consecutive primes squared hold the key to a furthered understanding of the distribution of primes; as evidenced, this perspective also builds strong support in favour of the Riemann hypothesis.
\end{abstract}

 \maketitle



\section{Introduction}
\label{sec:intro}
An important theme in analytic number theory is the distribution of primes in short intervals. Notably, it is still an open question for what exact interval lengths the prime number theorem is valid or breaks down, as it eventually does for short enough intervals. The problem, formally stated, is to identify for which functions $y=y(x)$, as $x\rightarrow \infty$, 
\begin{align}
	\pi(x + y)-\pi(x) \sim \frac{y}{\log x}.
\label{eq:shortPNT}
\end{align}

Currently, the gulf between the known upper and lower bounds of $y$ is huge. The best unconditional estimate of the upper bound, proved by Heath-Brown \citep{Heath-Brown:1988}, is $y=x^{7/12+\epsilon(x)}$, where $\epsilon(x)\rightarrow 0$ as $x\rightarrow \infty$, an estimate that is lowered to $y=x^{1/2+\epsilon}$ assuming the Riemann hypothesis (RH). In contrast, Selberg, also assuming RH, proved that as long as $y/ (\log x)^2 \rightarrow \infty$ when $x\rightarrow \infty$, then \eqref{eq:shortPNT} holds for \textit{almost all} $x$, understood in the sense that the set of $x\in[0,X]$ for which \eqref{eq:shortPNT} is not true is $o(X)$ as $X\rightarrow \infty$. Unconditionally, this almost-all result has been proven to hold with $y=x^{1/6-\epsilon}$ \citep{Zaccagnini:1998ub}. It was long thought that Selberg's result  would be true even in the case of all $x$ \citep{Granville:1995:2}, but Maier \citep{Maier:1985ww} made it clear that there are indeed infinitely many exceptions to Selberg's result, for any function of the form $y=(\log x)^\lambda$, with $\lambda>1$. It is not known whether there exist exceptions that are larger than $(\log x)^\lambda$, but it has been conjectured that $y> x^\epsilon$ will suffice for \eqref{eq:shortPNT} to hold for all $x$, see e.g. \citep{Granville:1995:2} or \citep{Soundararajan:2007vl}.

In this paper we draw attention to the short intervals between consecutive primes squared, defined by $s_k:=\{p_{k}^2, \dots p_{k+1}^2-1\}$ for $k\geq 1$. These intervals naturally occur in the context of the sieve of Eratosthenes, and in particular, each $s_k$ has the specific quality of being fully sieved by the $k$ first primes; any element in $s_k$ is either divisible by some $p \in \mathcal{P}_k:=\{p_1, \dots,p_{k}\}$ or else is a prime $p \notin \mathcal P_k$. In addition, the exact distribution of primes in $s_k$ is in its entirety build up of the periodic sequences 
\begin{align*}
\rho_{k}(n):=\begin{cases}
    p_k & \text{if $p_k \divides n$,}\\
    1 & \text{otherwise},
  \end{cases}
\end{align*}
which we visualise for the specific example of $s_3$ by the following table:\\
{\scriptsize
$$
\begin{tabular}{c|ccccccccccccccc|}
\obsgray{$n$} & 25 & 26 & 27 & 28 & \bf{29} & 30 & \bf{31} & 32 & 33 & 34 & 35 & 36 & \bf{37} & $\cdots$ & 48\\
\hline  \\ [-3mm]
\obsgray{$\rho_1(n)$}  & 1 & $p_1$ & 1 & $p_1$ & 1 & $p_1$ & 1 & $p_1$ & 1 & $p_1$ & 1 & $p_1$ & 1 &$\cdots$ & $p_1$\\
\obsgray{$\rho_2(n)$}  &1 & 1 & $p_2$ & 1 & 1 & $p_2$ & 1 & 1 & $p_2$ & 1 & 1 & $p_2$ & 1 &$\cdots$ & $p_2$\\
\obsgray{$\rho_3(n)$}  &$p_3$ & 1 & 1 & 1 & 1 & $p_3$ & 1 & 1 & 1 & 1 & $p_3$ & 1 & 1 &$\cdots$ & 1
\end{tabular}
$$}\\
Together with the fact that these intervals make up a complete subdivision of the natural numbers, these insights allow us to argue for the range in which the prime number theorem is valid, as well as providing a rudimentary explanation of the behaviour of the prime counting function $\pi(x)$. Specifically, we present a heuristic as to why the following conjecture should hold:
\begin{conjecture*}
The choice of $y= x^{1/2}$ is necessary and sufficient for 
\begin{align*}
	\pi(x + y)-\pi(x) \sim \frac{y}{\log (x+y)}
\end{align*}
to hold for all $x$ as $x\rightarrow \infty$.
\end{conjecture*}

%
%
%
Furthermore, we present evidence that the set of $\pi_k$s satisfies a conjecture of Montgomery and Soundararajan \citep{Montgomery:2004de}, that the number of primes in $s_k$, denoted $\pi_k$, is expected to follow a normal distribution with mean 
\begin{align*}
	\mu_k = \frac{|s_k|}{\log p_{k+1}^2}
\end{align*}
and standard deviation
\begin{align*}
	\sigma_k = 
		\frac{\sqrt{|s_k| \log (p_{k+1}^2/|s_k|)}}{\log p_{k+1}^2}.
\end{align*}
Because of the periodic structure of the sequences $\rho_{k}(n)$, the number of primes in two intervals $s_i$ and $s_j$ are not independent, but correlated. Following from this, we demonstrate that: 
\begin{rem2}
The prime counting function $\pi(x)$ behaves ---and should be understood---as a sum of \emph{correlated} random variables $\pi_k$, each normally distributed with mean $\mu_k$ and standard deviation $\sigma_k$.
\end{rem2}

Building on this insight, we formulate a random model of the primes, where the random variables are the number of primes in each interval $s_k$, and where these variables derive from drawing random translations of the sequences $\rho_{k}(n)$. This construction naturally preserves the observed correlations between $\pi_k$s. Moreover, assuming this model, we calculate the theoretical variance of $\pi(p_{k+1}^2)=\sum_{j=1}^k \pi_j$, which in turn suggests the conjecture
\begin{conjecture*}
The error term in the prime number theorem satisfies
\begin{align*}
	|\pi(x) - \li(x)| = O\left(\sqrt{\li(x)}\right).
\end{align*}
\end{conjecture*}

In sum, it emerges that the apparent random behaviour of the primes and the prime counting function $\pi(x)$ can be fundamentally understood in terms of the intervals $s_k$ and the underlying periodic sequences $\rho_k(n)$. Through Koch's equivalence criterium \citep{vonKoch:1901ui}, our results also strongly support the correctness of the Riemann hypothesis, and as such might pave the way towards a complete technical proof. 

%
%

\section{Notation and definitions}
\label{sec:notation}
We write the set consisting of the $k$ first primes as $\mathcal P_k:=\{p_1, \dots, p_k\}$, the intervals between consecutive primes squared as $s_k:=\{p_k^2, \dots, p_{k+1}^2-1\}$, and the length of $s_k$ as $l_k := |s_k| = p_{k+1}^2 - p_{k}^2$, where $k\geq 1$. As usual, the number of primes less than or equal to $x$ is given by $\pi(x)$, but in addition, since the number of primes in $s_k$ will appear frequently, we establish the shorthand notation
\begin{align*}
	\pi_k := \pi(p_{k+1}^2) - \pi(p_{k}^2).
\end{align*}
For the same reason, in the case when we apply the logarithmic integral, 
\begin{align*}
	\li(x):= \int_2^x \frac{dt}{\log t},
\end{align*}
to the interval $s_k$, we write
\begin{align*}
	\li_k:= \li(p_{k+1}^2) - \li(p_{k}^2).
\end{align*}

Moreover, we need notation for the expected number of primes in $s_k$---understood as the expected number of coprimes to $p_k\#$ in a random interval of length $l_k$ (where $p_k\#:=\prod_{p\in \mathcal P_k}p$ is the primorial of $p_k$). The probability of a random integer being coprime to $p_k\#$ is given by the Euler product $\prod_{p \in \mathcal P_k} \left(1- \frac{1}{p}\right)$, and multiplying this by $l_k$ produces the expected number of primes in $s_k$, which we denote 
\begin{align}
	\tilde \pi_k := l_k \cdot \prod_{p \in \mathcal P_k} \left(1- \frac{1}{p}\right).
\label{eq:tildepi_k2}
\end{align}

We now use this definition to construct a probabilistic prime counting function $\tilde \pi(x)$---the expected number of primes less than or equal to $x$. By assuming $k$ to be the integer such that $p_k^2 \leq x < p_{k+1}^2$, we define $\tilde \pi(x)$ simply as the sum taken over the individual estimates $\tilde \pi_j$, $1\leq j \leq k$, namely
\begin{align}
	\tilde \pi(x) :=  \sum_{j=1}^{k-1} \tilde \pi_j
			+ \frac{x-p_k^2}{l_k} \, \tilde \pi_k. 
\label{eq:tildepi_x}
\end{align}
The last term on the right side adjusts for the fact that $x$ in general reaches only partially into the last interval $s_k$.

Finally, we emphasise the fact that the distribution of primes within any interval $s_k$ can be viewed as a construction of $k$ periodic sequences.
To make this structure apparent,
we first define an arithmetic function that picks out the integers $n$ coprime to a given prime $p_k$,
\begin{align*}
\rho_{k}(n):=\begin{cases}
    p_k & \text{if $p_k \divides n$,}\\
    1 & \text{otherwise}.
  \end{cases}
\end{align*}
Then we apply this definition to construct a second arithmetic function that locates all $n$ coprime to $p_k\#$,
\begin{align*}
	  R_{k}(n) :=   \prod_{1\leq i \leq k} \rho_{i}(n).
\end{align*}
By these definitions $R_{k}(n)=1$ whenever $(n,p_k\#)=1$, and both $\rho_{k}(n)$ and $R_{k}(n)$ are periodic, satisfying for  any integer $m$ the equalities
\begin{align*}
	\rho_{k}(n + m p_k) = \rho_{k}(n) \quad \textrm{and} \quad 	R_k(n+ m p_k \#) = R_k(n).
\end{align*}

Since $s_k$ is sieved completely by the $k$ first primes, the primes within $s_k$ align with the 1s in $R_k$, and therefore
the number of primes in $s_k$ equals the number of 1s in $R_k$ across the interval. We visualise this for the specific example of $s_3$ by the following table:\\
{\scriptsize
$$
\begin{tabular}{c|ccccccccccccccc|}
\obsgray{$n$} & 25 & 26 & 27 & 28 & \bf{29} & 30 & \bf{31} & 32 & 33 & 34 & 35 & 36 & \bf{37} & $\cdots$ & 48\\
\hline  \\ [-3mm]
\obsgray{$\rho_1(n)$}  & 1 & $p_1$ & 1 & $p_1$ & 1 & $p_1$ & 1 & $p_1$ & 1 & $p_1$ & 1 & $p_1$ & 1 &$\cdots$ & $p_1$\\
\obsgray{$\rho_2(n)$}  &1 & 1 & $p_2$ & 1 & 1 & $p_2$ & 1 & 1 & $p_2$ & 1 & 1 & $p_2$ & 1 &$\cdots$ & $p_2$\\
\obsgray{$\rho_3(n)$}  &$p_3$ & 1 & 1 & 1 & 1 & $p_3$ & 1 & 1 & 1 & 1 & $p_3$ & 1 & 1 &$\cdots$ & 1\\
\hline  \\ [-3mm]
\obsgray{$R_3(n)$}  & $p_3$ & $p_1$ & $p_2$ & $p_1$ & \bf{1} & $p_1 p_2 p_3$ & \bf{1} & $p_1$ & $p_2$ & $p_1$ & $p_3$ & $p_1p_2$ & \bf{1} &$\cdots$ & $p_1p_2$ 
\end{tabular}
$$}\\
But in general, if we pick a random interval $\tilde s_k$ of length $l_k$, the 1s in $R_k$ across $\tilde s_k$ are not necessarily primes; all we know is that they are coprime to $p_k\#$. And because $R_k$ is periodic with period $p_k\#$, the number of coprimes to $p_k\#$ in $\tilde s_k$ can take any out of $p_k\#$ (non-unique) values; the probabilistic counting function $\tilde \pi_k$ is the mean value of these.

The last statement can be made explicit by applying sieve notation. Consider therefore an arbitrary interval $A$ and denote the number of coprimes to $p_k\#$ in $A$ by
\begin{align*}
	S(A,p_k\#) := |\{n: n\in A, R_k(n)=1 \}|.
\end{align*}
It follows from the periodicity of $R_k$ that $S(A,p_k\#)$ is periodic as well; if $A^j$ denotes $A$ left-shifted $j$ times, the equality
\begin{align*}
	S(A^{m\cdot  p_k\#},p_k\#) = S(A,p_k\#)
\end{align*}
holds for any integer m. Using this notation, it is obvious that we can restate the probabilistic prime counting function $\tilde \pi_k$ 
as the mean value taken over the sample space 
\begin{align*}
	\Omega_k := \{S(s_k^j,p_k\#)\}_{j=0}^{p_k\#-1}.
\end{align*}
In other words, 
\begin{align*}
	\tilde \pi_k = \frac{1}{p_k\#} \sum_{j=0}^{p_k\#-1}  S(s_k^j,p_k\#).
\end{align*}

The actual number of primes in $s_k$ coincides with the specific element in $\Omega_k$ corresponding to $j=0$,  
\begin{align*}
	\pi_k = S(s_k^0,p_k\#)= S(s_k,p_k\#).
\end{align*}
The prime counting function $\pi_k$ is therefore just one instance out of $p_k\#$ elements in the sample space underlying $\tilde \pi_k$. The crucial point, nonetheless, is that all elements in $\Omega_k$, $\pi_k$ inclusive, stem from the same underlying structure of $k$ periodic sequences.

\section{Approximate expressions for $\tilde \pi_k$ and $\tilde \pi(x)$}
\label{sec:probPNT}
From \eqref{eq:tildepi_k2}, we know that the expected number of primes in the interval $s_k$ can be expressed exactly in terms of  
\begin{align*}
	\tilde \pi_k = l_k \cdot \prod_{p \in \mathcal P_k} \left(1- \frac{1}{p}\right).
\end{align*}
Applying Merten's product theorem, we attain the following approximation:  
\begin{lem}
\label{lem:MPNT1} 
\begin{align}
	\tilde \pi_k \sim  2 \e^{-\gamma} \frac{l_k}{\log p_{k+1}^2}.
\label{eq:tildepi_k_mertens}
\end{align}
\end{lem}

\begin{proof} Assume that $x$ is a real number within the interval $s_k$, so that $p_k^2 \leq x < p_{k+1}^2$. Then we can state Merten's product theorem 
\citep{Mertens:1874tx}
as
\begin{align}
	\prod_{p \in \mathcal  P_k} (1-\frac{1}{p}) = \e^{-\gamma+\delta}\frac{1}{ \log \sqrt{x}} = 2 \e^{-\gamma+\delta}\frac{1}{ \log x}, 
\label{eq:mertens}	
\end{align}
where $\gamma$ is the Euler--Mascheroni constant and $\delta$ is a measure of the uncertainty of the approximation, satisfying
\begin{align}
	|\delta | < \frac{4}{\log (\sqrt{x}+1)} + \frac{2}{\sqrt{x} \log \sqrt{x}} + \frac{1}{2 \sqrt{x}}.  
\label{eq:mertenerror}
\end{align}
By combining \eqref{eq:tildepi_k2} and \eqref{eq:mertens}, it follows immediately that 
\begin{align*}
	\tilde \pi_k =  2 \textrm{e}^{-\gamma+\delta} \frac{l_k}{\log x}.
\end{align*}
Additionally, \eqref{eq:mertenerror} implies that in order to minimize the error we should choose $x$ as large as possible---that is, $x=p_{k+1}^2-\epsilon$, where $\epsilon>0$ is infinitesimal---by which we obtain \eqref{eq:tildepi_k_mertens}.
\end{proof}
Note that by the choice of $x=p_{k+1}^2$ in the proof above, the estimate of $\tilde \pi_k$ reflects the characteristic property of the sieving process; sieving by the $k$ first primes removes all composites less than $p_{k+1}^2$. 

Similarly, we obtain an approximate expression for $\tilde \pi(x)$:  
\begin{lem}
\label{lem:MPNT3}
\begin{align*}
	\tilde \pi(x) \sim  2 \e^{-\gamma} \li (x).
\end{align*}
\end{lem}

\begin{proof} 
First, applying \lemref{lem:MPNT1} to \eqref{eq:tildepi_x} we immediately have that 
\begin{align*}
	\tilde \pi(x) 
	\sim 2 \e^{-\gamma} \left(  \sum_{j=1}^{k-1} \frac{l_j}{\log p_{j+1}^2}
			+  \frac{x-p_k^2}{l_k} \, \frac{l_k}{\log p_{k+1}^2} \right).
\end{align*}
The expression within parentheses is nothing but a Riemann sum, which in the continuum limit can be approximated by the logarithmic integral $\li(x)$: The logarithmic integral taken over the interval $s_k$ satisfies
\begin{align*}
	\frac{l_k}{\log {p_{k}^2} } > \li_k > \frac{l_k}{\log {p_{k+1}^2} }, 
\end{align*}
and also,
\begin{align*}
	\frac{l_k}{\log {p_{k}^2} }  \sim \frac{l_k}{\log {p_{k+1}^2} }.
\end{align*}
We therefore have that 
\begin{align*}
	\tilde \pi_k \sim 2 \e^{-\gamma} \li_k,
\end{align*}
and eventually,
\begin{align*}
	\tilde \pi(x) 
	\sim 2 \e^{-\gamma} \left(  \sum_{j=1}^{k-1} \li_j
			+  \frac{x-p_k^2}{l_k} \, \li_k \right)
	= 2 \e^{-\gamma}  \li(x).			
\end{align*}
\end{proof}

We emphasise here again that the above results are derived under the sole assumption that $s_k$ is sieved by the $k$ first primes. In this case, the best we can do is to assign a uniform probability of finding a prime in any position across $s_k$. This probability is given exactly by $\prod_{p \in \mathcal P_k} \left(1- \frac{1}{p}\right)$, or approximately by $2 \e^{-\gamma}/\log p_{k+1}^2$, Obviously, there is more to say about $s_k$---illustrated with a few examples in the next section---and this additional information is what eventually will close the gap between $\tilde \pi_k$ and $\pi_k$,
where the latter is anticipated to satisfy 
\begin{align*}
	\pi_k \sim \frac{l_k}{\log p_{k+1}^2}.
\end{align*}

It is also worth noting that the expression $2 \e^{-\gamma} \li(x)$ emerges as the continuum approximation to the discrete sum $2 \e^{-\gamma} \sum_{j=1}^k l_j/\log p_{j+1}^2$, and not the other way around. This fact hints to the possibility that even the prime counting function $\pi(x)$ is best approximated by $\sum_{j=1}^k l_j/\log p_{j+1}^2$ rather than $\li(x)$. We examine this in closer detail in \secref{sec:numresults}.

\section{Shrinking the gap between $\tilde \pi_k$ and $\pi_k$}
\label{sec:shrinkGap}
Recall from \secref{sec:notation} that the probabilistic prime counting function $\tilde \pi_k$ can be stated as 
\begin{align*}
	\tilde \pi_k =  \frac{1}{p_k\#} \sum_{j=0}^{p_k\#-1}  S(s_k^j,p_k\#),
\end{align*}
while $\pi_k$ is given by
\begin{align*}
	\pi_k = S(s_k^0,p_k\#).
\end{align*}
To close in on an estimate for $\pi_k$, we need to shrink the sample space $\Omega_k$ in such a way that it still contains $S(s_k^0,p_k\#)$. Of course, $S(s_k^0,p_k\#)$ is completely determined by where in the natural numbers the interval $s_k$ is situated (allowing for shifts that are multiples of $p_k\#$), so the constraints on $\Omega_k$ must reflect this fact. While defining the right constraints is an essential part of sieve theory, our point here is only to illustrate with a few numerical examples how imposing constraints affect the probabilistic estimate $\tilde \pi_k$.

The most obvious constraint is that all elements in $R_k(n)$ must be strictly smaller than $p_{k+1}^2$ whenever $n\in s_k$. We can observe the effect of this constraint by expanding the Euler product \eqref{eq:tildepi_k2} in terms of the the M\"obius function, which is defined by 
\begin{align}
\mu(n):=
\begin{cases}
1 \quad \textrm{if } n=1, \\
(-1)^k \quad \textrm{if $n$ is a product of $k$ distinct primes},\\
0 \quad \textrm{if $n$ has one or more repeated prime factors}.
\end{cases}
\label{eq:mobius}
\end{align}
%
It follows that we can rewrite \eqref{eq:tildepi_k2} as
\begin{align*}
	\tilde \pi_k = 
	l_k  \cdot \sum_{ \substack{ {d \mid p_k\#}} } \frac{\mu(d)}{d}.
\end{align*}
Adding the constraint produces a truncated version, which we denote
\begin{align*}
	\tau_{k} = 
	l_k  \cdot \sum_{ \substack{ {d \mid p_k\#}\\{d < p_{k+1}^2}} } \frac{\mu(d)}{d}
\end{align*}
Numerically---as seen in \figref{fig:legendreTruncated}---we verify that the truncated expression amounts to a significant reduction of the sample space $\Omega_k$, resulting in 
$\tau_{k}/(l_k/\log p_{k+1}^2) \approx 1.03$, as opposed to $\tilde \pi_k/(l_k/\log p_{k+1}^2)\sim 2\e^{-\gamma} \approx 1.12$. In addition, the error term, which is $O(2^k)$ in the case of $\tilde \pi_k$, reduces to $O(k^{2.32})$ in the case of $\tau_{k}$.

\begin{figure}[h]
	\centering
	\includegraphics[width=250pt]{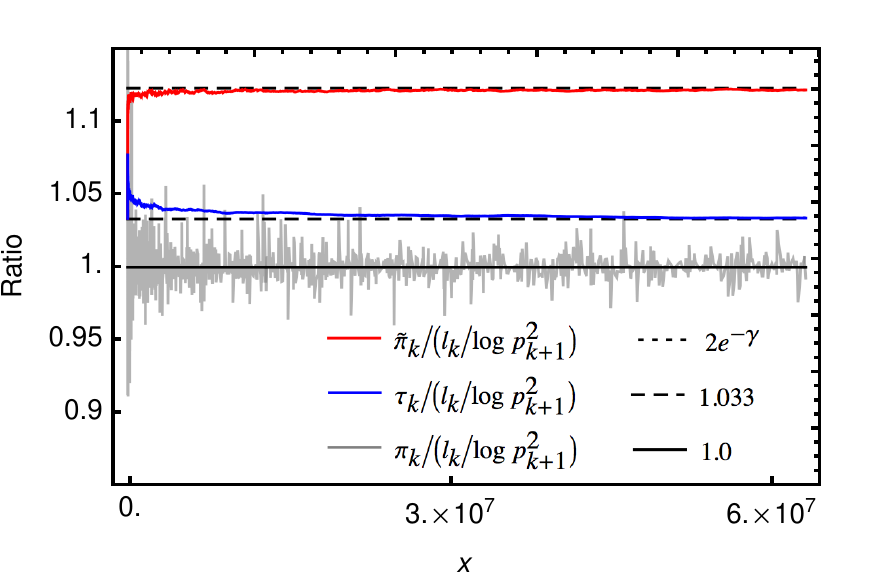}
\caption{
\small \sl
The ratios of $\pi_k$, $\tilde \pi_k$, and $\tau_k$ to $l_k/\log p_{k+1}^2$. The values are plotted at $x=p_{k+1}^2$, $1\leq k \leq 1000$. The three horizontal lines illustrate the limits tended to by each ratio. 
}
\label{fig:legendreTruncated}
\end{figure}
Our second example is a set of constraints on the internal structure of $s_k$: Within $s_k$, the first appearance of a composite divisible by $p_i$, for $i<k$, is in one of the positions $p_i - m+1$, where $m>0$ is an element in the residue class of $n^2 \imod{p_i}$, $j> i$ ($m=0$ is not included, since  the first position in $s_k$ is always occupied by $p_k^2$). For instance, given the primes $p_1,\dots,p_6$, these are the possible positions of first appearance when $k>6$:
\vspace{2mm}
\begin{center}	
\begin{tabular}{cl}
Prime              		& Candidate positions of first appearance in $s_k$\\
\hline
$p_1$ 				& 2  \\
$p_2$ 				& 3  \\
$p_3$				& 2, 5  \\
$p_4$ 				& 4, 6, 7  \\
$p_5$ 				& 3, 7, 8, 9, 11  \\
$p_6$ 				& 2, 4, 5, 10, 11, 13 
\vspace{1.5mm}
\end{tabular}
\end{center}
The outcome is that these constraints bound the possible positions of sequences $\rho_i(n)$ across $s_k$, $1\leq i \leq k$, effectively halving the sample space $\Omega_k$. 

To get an impression of how these constraints impact the probabilist prime counting function $\tilde \pi_k$, we randomly sample sequences of $\rho_i$ of length $l_k$, $1\leq i \leq k$, with first appearance of $p_i$ in each sequence constrained as explained above, and count the number of 1s in the resulting $R_k$---denoted $\eta_k$. We observe in \figref{fig:residues} that the mean value of $\eta_k$ in fact appears to lie slightly above $\tilde \pi_k$, while still satisfying $\eta_k \sim \tilde \pi_k$. Note that in the figure we plot $\sum_{j=1}^k \eta_j$ rather than $\eta_k$ as the accumulation of biases is more visible than the individual biases in each interval. It appears therefore that the statistical properties of the reduced sample space do not change significantly from the original one, which can be understood from the fact that these constraints do not restrain the position of $s_k$ within the natural numbers.

\begin{figure}[h]
	\centering
	\includegraphics[width=358pt]{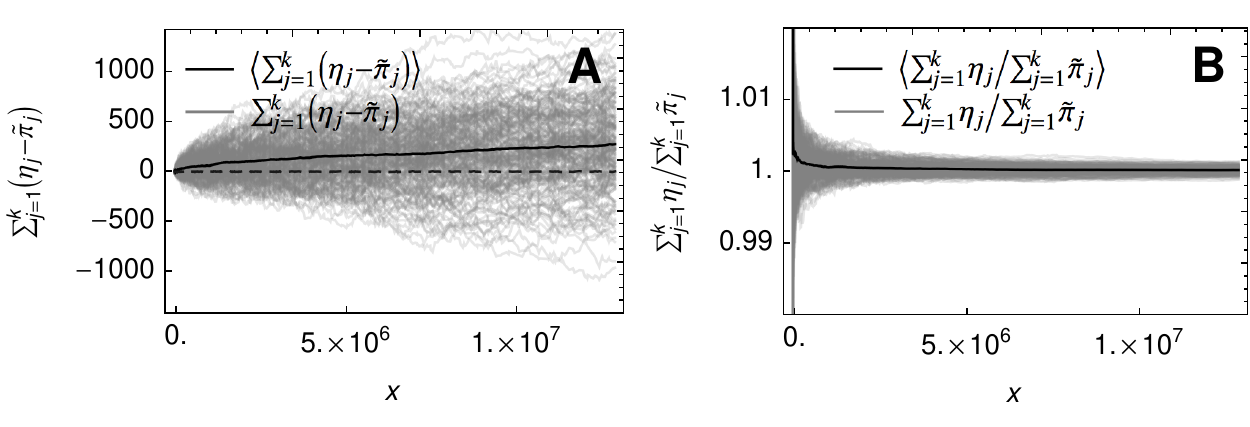}
\caption{
\small \sl
Comparison of $\sum_{j=1}^k \eta_j$ and $\sum_{j=1}^k \tilde \pi_j$, with values plotted at $x=p_{k+1}^2$, $1\leq k \leq 500$. A) 200 samples of $\sum_{j=1}^k (\eta_j - \tilde \pi_j)$ (grey) shown together with the mean value $\langle \sum_{j=1}^k (\eta_j - \tilde \pi_j)\rangle$ (black), demonstrating a statistical bias of 
$\eta_k > \tilde \pi_k$.
 B) 200 samples of $\sum_{j=1}^k  \eta_j / \sum_{j=1}^k \tilde \pi_j$ (grey) shown together with the mean value $\langle \sum_{j=1}^k  \eta_j / \sum_{j=1}^k \tilde \pi_j \rangle$ (black), indicating that  $\eta_k \sim \tilde \pi_k$.
}
\label{fig:residues}
\end{figure}

\section{Conjectures on the distribution of primes in short intervals}
\label{sec:newPNT}
The prime number theorem, proven independently by Hadamard and de la Vall\'{e}e-Poussin in 1896, is usually stated by either of the equivalent relations
\begin{align*}
	\pi(x) \sim \frac{x}{\log x} \quad \textrm{or} \quad \pi(x) \sim \li(x). 
\end{align*}

The results so far, however, motivates an alternative formulation of the prime number theorem, in the form of the conjecture:
\begin{conjecture}
\label{con:pntlocal}
As $k \rightarrow \infty$, we have that
\begin{align*} 
	\pi_k \sim  \frac{l_k}{\log p_{k+1}^2}.
\end{align*}
\end{conjecture}
From this perspective---as in the probabilistic case in \secref{sec:probPNT}---we should view $\pi(x) \sim \li(x)$ as originating from a continuum approximation of $\sum_{j=1}^k l_j / \log p_{j+1}^2$. 

Proving \conref{con:pntlocal} rigorously is out of reach for this paper, but we present a heuristic justification that hopefully can seed a complete proof. Specifically, we argue that the following conjecture holds, which in turn implies \conref{con:pntlocal}: 
\begin{conjecture}
\label{con:pntshort} 
The choice of $y= x^{1/2}$ is necessary and sufficient for 
\begin{align}
	\pi(x + y)-\pi(x) \sim \frac{y}{\log (x+y)}
\label{eq:pntshort}
\end{align}
to hold for all $x$ as $x\rightarrow \infty$.
\end{conjecture}
To start with, let us consider what happens to an interval that grows slower than $y=x^{1/2}$. For this purpose we define $y_{\epsilon}=y_{\epsilon}(x)$ to be a function that grows arbitrarily slower than $y$. Naturally, no matter how large $y_{\epsilon}$ is initially, as $x$ increases, $y_{\epsilon}$ eventually becomes infinitesimal compared to $y$. This has the effect that for any given integer $q>0$, we can find an $x$ such that $y_{\epsilon} \leq y/2^q$. It then follows from Bertrand's postulate---which says there is always a prime $p$ such that $n<p<2n$---that 
\begin{align*}
	\pi(y) - \pi(y_{\epsilon})  \geq \pi(y) - \pi(y/2^q) = q.
\end{align*}
And because $q$ can be arbitrarily large, 
\begin{align*}
	\lim_{x \rightarrow \infty}\pi(y) - \pi(y_{\epsilon}) = \infty.
\end{align*}
Therefore, across $s_k$ we will always find that $p_k \leq y< p_{k+1}$, while $y_{\epsilon}$ lags behind and ultimately turns infinitesimal relative to infinitely many primes smaller than $p_k$. 

Let us explore the consequence of the last statement: As described earlier, an essential property of the distribution of primes in $s_k$ is that we can view it as a construction of the periodic  sequences $\rho_j(n)$, $1\leq j \leq k$. From this perspective, it is clear that if we want to sample correctly the distribution of primes within $s_k$, we need to choose y such that the interval $[x,x+y]$ spans at least the largest underlying period, that is, $y \geq p_k$ whenever $x\in s_k$; the slowest growing function that satisfies this criteria is $y= x^{1/2}$. The interval $[x,x+y_\epsilon]$, on the other hand, is in general only long enough to provide valid samples from the distribution of coprimes to $p_m\#$, where $p_m$ is the greatest prime smaller or equal to $y_{\epsilon}$ across $s_k$. Since, from our argument above, $p_m$ eventually grows infinitesimal relative to arbitrarily many primes smaller than $p_k$, we should expect $[x, x+y_\epsilon]$ to ultimately move repeatedly across regions where the density of primes deviates significantly from that predicted by the prime number theorem. 

What this suggests is that $y=x^{1/2}$ is the sharp barrier below which the prime number theorem breaks down for all $x$, contradicting previous conjectures that this barrier lies close to $y=x^{\epsilon}$, $\epsilon >0$  \citep{Granville:1995:2, Soundararajan:2007vl}.
Furthermore, 
%
%
%
taking this reasoning to its conclusion, it seems logical to conjecture a variation of Maier's theorem \citep{Maier:1985ww} applied to any interval growing slower than $x^{1/2}$:
\begin{conjecture}
\label{con:Maier} 
Let $y_\epsilon=y_\epsilon(x)$ satisfy $\lim_{x\rightarrow \infty} y_\epsilon /y = 0$, where $y=x^{1/2}$. Then
\begin{align*}
\limsup_{x\rightarrow \infty} \frac{\pi(x+y_\epsilon) -\pi(x)}{y_\epsilon/ \log x} > 1
\qquad \textrm{and} \qquad
\liminf_{x\rightarrow \infty} \frac{\pi(x+y_\epsilon) -\pi(x)}{y_\epsilon/ \log x} < 1.
\end{align*}
%
\end{conjecture}

So far, we have argued that $y=x^{1/2}$ is a necessary condition for \eqref{eq:pntshort} to hold for all $x$. To substantiate that $y=x^{1/2}$ is also sufficient, we make the observation that the lengths of the intervals $s_k$ are of the form $l_k=2 p_{k+1} g_k-g_k^2$, where $g_k:=p_{k+1}-p_k$, and hence lie on the curves $2 \sqrt{x} g - g^2$, with $g=2n$, $n\geq 1$. Therefore, $l_k$ grows as $O(x^{1/2})$. Let us further define $y_\delta = y_\delta(x)$ to be a function that grows arbitrarily faster than $y=x^{1/2}$, so that $\lim_{x\rightarrow \infty} y_\delta / y = \infty$. As a consequence, the interval $[x,x+y_\delta]$ will eventually cover arbitrarily many intervals $s_k$; that is, for any $m$, we can always choose $k$ so that $y_\delta(p_k^2) \geq \sum_{i=k}^{k+m} l_i $. The problem we encounter now, however, is that the estimate 
\begin{align*}
	\pi(p_k^2+ y_\delta) - \pi(p_k^2) \sim \frac{ y_\delta}{\log(p_k^2+ y_\delta)}
\end{align*}
assumes a uniform distribution of primes in the interval $[p_k^2,p_k^2+y_\delta]$, given by the density of primes in $s_m$. But this assumption is inaccurate; the largest intervals across which we can suppose a uniform distribution are the intervals $s_k$. Subdividing the natural numbers into intervals of length $y_\delta$ and estimating the number of primes up to $x$ from these results in an underestimate of the prime counting function $\pi(x)$. At its most extreme, this is exemplified by the estimate $\pi(x) \sim x/\log x$, which is well known to be an inferior guess of the number of primes up to $x$ compared to $\pi(x) \sim \li(x)$, a fact that was established in 1899 by de la Vall\'{e}e-Poussin \citep{Poussin:1899}. 

One final note about this heuristic. We have only argued for what lengths $y$ can take, in order for the interval $[x,x+y]$ to correctly sample the distribution of primes, but not what is the correct limiting value of $\pi(x+y) - \pi(x)$. The asymptotic limit in \conref{con:pntshort} should therefore be considered as an assumption of the heuristic as well. 

\section{Numerical results on the distribution of primes}
\label{sec:numresults}

With the purpose of providing clues about how to understand the apparent random behaviour of the primes, we here examine the distribution of primes both locally and globally. We start by noting that since $l_k = 2 p_{k+1} g_k - g_k^2$, where $g_k = p_{k+1} - p_k$, we have that 
\begin{align*}
	\pi_k \sim \frac{l_k}{\log p_{k+1}^2} = \frac{2 p_{k+1} g_k - g_k^2}{\log p_{k+1}^2}.
\end{align*}
Naturally, $g_k$ is not unique, suggesting that we can subdivide the set of intervals $s_k$ according to what the corresponding gap $g_k$ is. Hence, the subset of intervals $s_k$ all with corresponding gap $g$ will be distributed along the curve
\begin{align*}
	\frac{2 \sqrt{x} g - g^2}{\log x},
\end{align*}
where $g$ takes the values $2n$, $n\geq 1$. This structure becomes apparent when plotting the values of $\pi_k$ as a function of $p_{k+1}^2$, as illustrated in \figref{fig:pikdistribution}, where all values of $\pi_k$ are plotted for $1\leq k \leq 6 \times 10^5$.

%
%
%
%
%

While this perspective clearly reveals how the intervals $s_k$ define an underlying pattern of the distribution of primes, it is possible to provide an even more compact viewpoint. According to a conjecture of Montgomery and Soundararajan \citep{Montgomery:2004de}, we should expect the primes in $s_k$ to follow a normal distribution with mean 
\begin{align*}
	\mu_k = \frac{l_k}{\log p_{k+1}^2}
\end{align*}
and standard deviation
\begin{align*}
	\sigma_k = 
		\frac{\sqrt{l_k (\log (p_{k+1}^2/l_k)+B)}}{\log p_{k+1}^2},
\end{align*}
where $B=1 - \gamma - \log 2\pi$. Actually, 
\begin{align*}
	\frac{\sqrt{l_k (\log (p_{k+1}^2/l_k)+B)}}{\log p_{k+1}^2}
	\sim 
	\frac{\sqrt{l_k \log (p_{k+1}^2/l_k)}}{\log p_{k+1}^2},
\end{align*}
but the latter expression does not give the same numerical accuracy for the data we have available, so we stick with $\sigma_k$ as defined. 

Next, we apply Montgomery and Soundararajan's conjecture to generate normalised versions of the prime counting functions $\pi_k$: 
\begin{align*}
	\overline{\pi}_k := \frac{\pi_k - \mu_k}{\sigma_k}.
\end{align*}
Assuming the conjecture holds, the result of this step should be that $\overline{\pi}_k$ is normally distributed with mean 0 and standard deviation 1. This expectation is accurately confirmed by \figref{fig:mongomery}, where we (A) plot $\overline{\pi}_k$ as a function of $k$ and (B) display a histogram revealing the probability distribution of $\overline{\pi}_k$. As such, this result lends support to Montgomery and Soundararajan's conjecture. But more importantly, it promotes the fundamental idea that we can view $\pi_k$ as a random variable with variance $\mu_k$ and standard deviation $\sigma_k$.

\begin{figure}
\centering
\includegraphics[width=8.7cm]{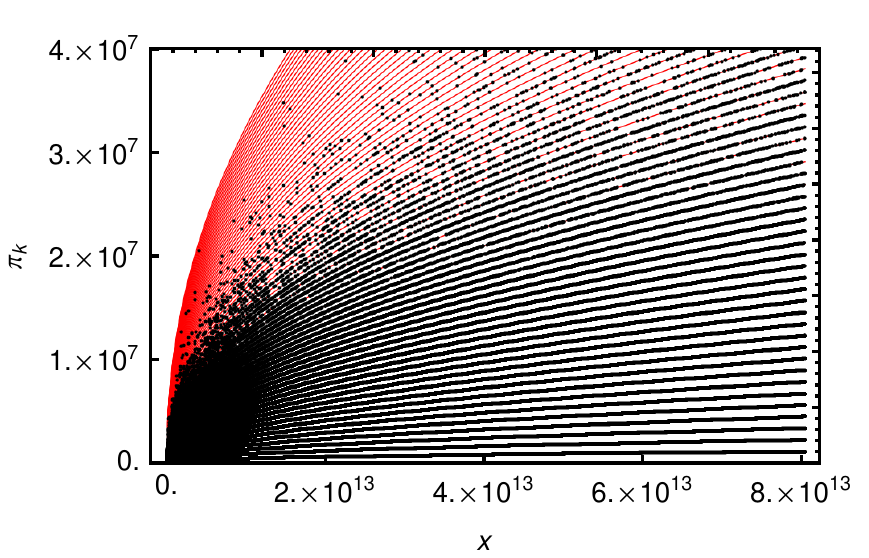}
\caption{
\small \sl
The number of primes $\pi_k$ in each interval $s_k$ plotted at the values $x = p_{k+1}^2$,  $1\leq k \leq 6 \times 10^5$ (black dots). The emerging curves relate to the specific gap values $g$, $g\in\{2,4,\dots\}$, where the bottom curve corresponds to the set of intervals $s_k$ where $g_k=2$, the second bottom $g_k=4$, etc. The red curves show the corresponding theoretical estimates $(2\sqrt{x}g - g^2)/ \log {x}$.
}
\label{fig:pikdistribution}
\end{figure}
\begin{figure}
\centering
\includegraphics[width=360pt]{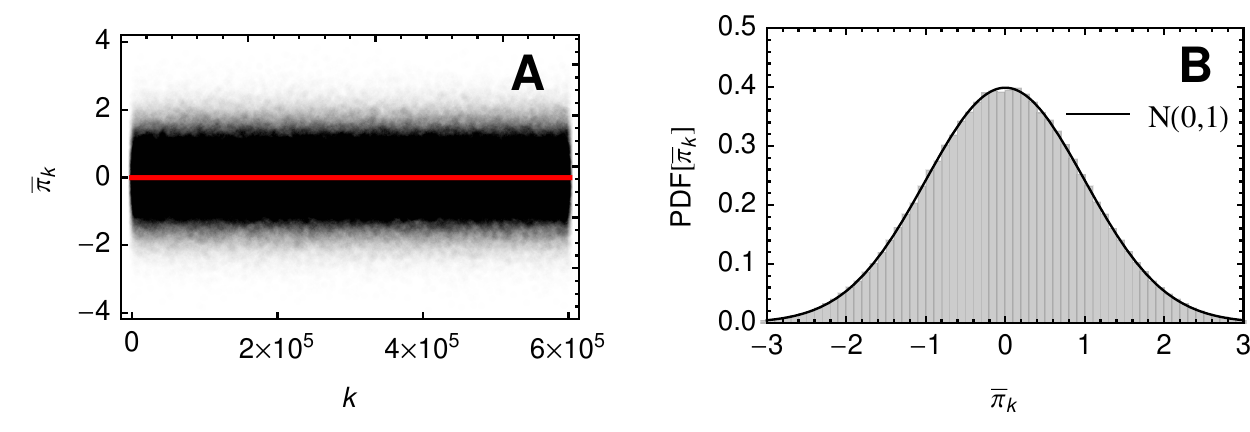}
\caption{
\small \sl
(A) The normalised prime counting functions $\overline{\pi}_k$ plotted as a function of $k$,  $1\leq k \leq 6 \times 10^5$. (B) A histogram displaying the numerical probability distribution of $\overline{\pi}_k$, overlaid with the PDF of the normal distribution $N(0,1)$ (black). The measured mean and standard deviation of $\overline{\pi}_k$ are $0.000(7)$ and $1.000(8)$, respectively.
}
\label{fig:mongomery}
\end{figure}


We turn now to the global distribution of primes, where we want to examine the behaviour of the prime counting function $\pi(x)$ when it is expressed in terms of the individual prime counting functions $\pi_k$:
\begin{align*}
\pi(p_{k+1}^2) = \sum_{j=1}^k \pi_j.
\end{align*}
To best get an impression of how $\pi(x)$ behaves, we plot---rather than $\pi(x)$ itself---the error function
\begin{align*}
	\epsilon(x):= \pi(x) - \li(x),
\end{align*}
which produces the result seen in \figref{fig:pix}. The first observation we make is that the error function starts out being negative---a famed bias that continues for the full stretch of our data set. Nonetheless, as Littlewood proved \citep{Littlewood:1914}, if we persist in moving $x$ towards infinity, the error term will eventually shift sign, and it will continue to do so infinitely many times over. When that happens for the first time, however, no one knows, though it has been proved that $x=1.39822\times 10^{316}$ is an upper bound for the event \citep{Bays2000}. A probabilistic measure of the bias was provided in 1994, when Rubinstein and Sarnak \citep{Rubinstein:1994vp}---assuming the Riemann hypothesis--- calculated the logarithmic density of those $x$ such that $\pi(x) - \li(x)>0$ to be around $0.00000026$.

\begin{figure}
\centering
\includegraphics[width=8.7cm]{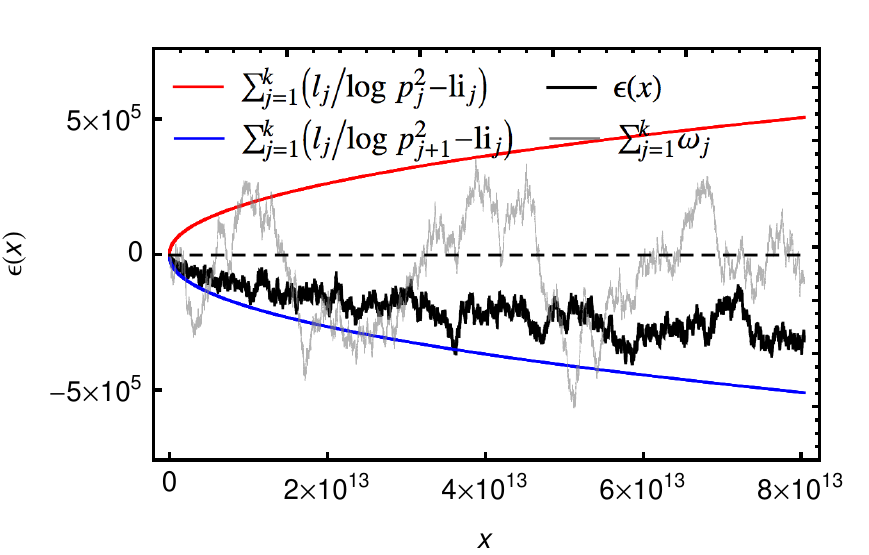}
\caption{
\small \sl
The error function $\epsilon(x)$ plotted at the values $x = p_{k+1}^2$, $1\leq k \leq 6\times10^5$ (black). Also shown are the differences 
$\sum_{j=1}^k (l_j/\log p_{j}^2 - \li_j)$ (red), and $\sum_{j=1}^k (l_j/\log p_{j+1}^2 - \li_j)$ (blue), as well as one realisation of the sum $\sum_{j=1}^k \omega_j$ (grey), where the $\omega_k$s are independent random variables each satisfying $N(0,\sigma_k)$.
}
\label{fig:pix}
\end{figure}

According to \conref{con:pntlocal}, however, $\li(x)$ is a continuum approximation to the discrete sum $\sum_{j=1}^k  l_j/\log {p_{j+1}^2}$, and therefore we should expect the latter to be a better estimate of $\pi(x)$. As we can tell from \figref{fig:pix}, where we also have included curves of $\sum_{j=1}^k  (l_j/\log {p_{j}^2} - \li_j)$ and $\sum_{j=1}^k  (l_j/\log {p_{j+1}^2} - \li_j)$, the statistical evidence seems to indicate that the truth lies somewhere in between,  
\begin{align*}
	\sum_{j=1}^{k} \frac{l_j}{\log {p_{j+1}^2} }<E[\pi(p_{k+1}^2)]<\li(p_{k+1}^2).
\end{align*}

To more accurately interpret this observation, we normalise the curves in \figref{fig:pix} by 
\begin{equation*}
	\Delta_k :=  \frac{1}{2}\sum_{j=1}^{k} \left( \frac{l_j}{\log {p_{j}^2} } - \frac{l_j}{\log {p_{j+1}^2} }\right),
\end{equation*}
which places the normalised differences 
\begin{equation*}
	\left[\sum_{j=1}^{k} \left( \frac{l_j}{\log {p_{j+1}^2} } - \li_j \right)\right]/{\Delta_k} 
	\quad \textrm{and} \quad 
	\left[\sum_{j=1}^{k} \left( \frac{l_j}{\log {p_{j}^2} } - \li_j \right)\right]/{\Delta_k} 
\end{equation*}
approximately at the constant lines -1 and 1 respectively. We also normalise the error function, defined by
\begin{align*}
	E_k:= \frac{\epsilon(p_{k+1}^2)}{\Delta_k}.
\end{align*}
The resulting curves are shown in \figref{fig:pixnormalized}A, plotted as functions of $k$. We note that on the scale of the data available, $E_k$ fluctuates fairly stable around a mean value of $-0.60$. This becomes even more apparent if we display the data in histogram form, as done in \figref{fig:pixnormalized}B, where we see that the empirical probability distribution of $E_k$ resembles a normal distribution. A natural speculation therefore, is whether there in fact  exists a constant $c$ such that 
\begin{align*}
	\E\left[\pi\left(p_{k+1}^2\right)\right]=c\cdot \li\left(p_{k+1}^2\right)-(1-c)  \cdot  \sum_{j=1}^k \frac{l_j}{\log {p_{j+1}^2} }.
\end{align*}

\begin{figure}
\centering
\includegraphics[width=360pt]{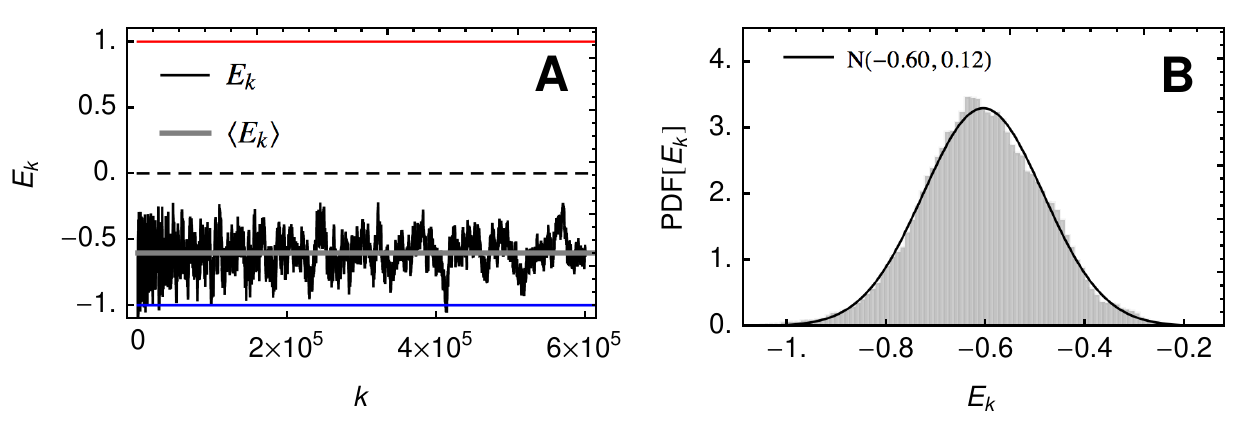}
\caption{
\small \sl 
The normalised error function $E_k$ plotted as a function of $k$, $1\leq k \leq 6\times10^5$ (A), as well as a probability density histogram of all values of $E_k$ (B). Also displayed in (B) is the PDF for a normal distribution with mean $-0.60(4)$ and standard deviation $0.12(1)$ estimated from all $E_k$.
}
\label{fig:pixnormalized}
\end{figure}

Returning to Figure 5, we end with one final important observation: The error function $\epsilon(x)$ does not conform to a sum of independent random variables. This is clearly visible when comparing with an example of an actual sum of independent random variables $\omega_k$, where 
$\omega_k\sim N(0,\sigma_k)$; the error function does not display the same scale of fluctuations as $\sum_k \omega_k$, and as already discussed, lingers more densely around a mean value. As will be apparent later, what causes the discrepancy between the error function $\epsilon(x)$ and $\sum_{j=1}^k \omega_j$ are correlations between the prime counting functions $\pi_k$. These correlations even grow stronger for larger interval lengths $l_k$, possibly explaining why the normalised error function $E_k$ in \figref{fig:pixnormalized}A seems to fluctuate more slowly as $k$ becomes large. Explaining these correlations and their effect on the error function will be the focus of the next sections. But for now, we summarise the numerical insights by the following remark:

\begin{rem}
The prime counting function $\pi(x)$ behaves---and should be understood---as a sum of \emph{correlated} random variables $\pi_k$, each normally distributed with mean $\mu_k$ and standard deviation $\sigma_k$ as conjectured in \citep{Montgomery:2004de}.
\end{rem}

\section{A random model for the distribution of primes}
\label{sec:randommodel}
With the understanding we have developed so far, we now have in hand enough information to construct a random model of the prime counting function $\pi(x)$. We start with emphasising the fundamental observation underlying the model; the exact distribution of primes up to $p_{k+1}^2$ is in its entirety build up of the $k$ periodic sequences $\rho_j(n)$, $1\leq j \leq k$, each having the property that $\rho_j(0)=p_j$. We visualise this in the following table: 
{\scriptsize
$$
\begin{tabular}{c|ccccccccccccccccc}
\obsgray{$n$} & \bf{0} & 1 & 2 & 3 & 4 & 5 & 6 & 7 & 8 & 9 & 10 & 11 & 12 & 13 & 14 & 15 & $\cdots$\vspace{1mm}\\
\hline  \\ [-3mm]
\obsgray{$\rho_1(n)$}  &$\boldsymbol{p_1}$ & 1 & $p_1$ & 1 & $p_1$ & 1 & $p_1$ & 1 & $p_1$ & 1 & $p_1$ & 1 & $p_1$ & 1 & $p_1$ & 1 &  $\cdots$\\
\obsgray{$\rho_2(n)$}  &$\boldsymbol{p_2}$ & 1 & 1 & $p_2$ & 1 & 1 & $p_2$ & 1 & 1 & $p_2$ & 1 & 1 & $p_2$ & 1 & 1 & $p_2$ & $\cdots$ \\
\obsgray{$\rho_3(n)$}  &$\boldsymbol{p_3}$ & 1 & 1 & 1 & 1 & $p_3$ & 1 & 1 & 1 & 1 & $p_3$ & 1 & 1 & 1 & 1 & $p_3$ &$\cdots$\\ 
\obsgray{\vdots}&\vdots&\vdots&\vdots&\vdots&\vdots&\vdots&\vdots&\vdots&\vdots&\vdots&\vdots&\vdots&\vdots&\vdots&\vdots&\vdots&
$\cdots$\\
\obsgray{$\rho_k(n)$}  &$\boldsymbol{p_k}$ & 1 & 1 & 1 & 1 & 1 & 1 & 1 & 1 & 1 & 1 & 1 & 1 & 1 & 1 & 1 &$\cdots$
\end{tabular}
$$\\
}
\noindent Importantly---and already discussed at length---within $s_k$, only the $k$ first sequences $\rho_j$, $1\leq j \leq k$, are needed to completely determine the distribution of primes. Note also that the distribution of primes in $s_k$ correlates with the distribution of primes in all previous intervals $s_i$, $1\leq i < k$, due to the periodicities of $\rho_j$, $1\leq j \leq i$.

In formulating the model---which we hereafter refer to as the {\it correlated random model}---the essential idea is to keep the intervals $s_k$, but to allow translations of the sequences $\rho_j$, $1\leq j \leq k$. By this approach, we stay close to the original structure of the primes, but will eventually be able to get a theoretical grip on the correlations between intervals; as becomes apparent, this is exactly what we need to give a proper explanation of the error function $\epsilon(x)$.  

The central object in the model is the random variable $\tilde \Pi(p_{k+1}^2)$, which is defined in terms of a sum over the random variables $\tilde \Pi_j$:
\begin{align*}
	\tilde \Pi(p_{k+1}^2) := \sum_{j=1}^k \tilde \Pi_j.
\end{align*}
To generate the random variables $\tilde \Pi_j$, $1\leq j \leq  k$, we randomly choose $k$ corresponding integers $m_j$, where $m_j \in \{0,1, \dots, p_j-1\}$ and then define the translated sequences
\begin{align*}
\hat{\rho}_{j}(n):=\rho_k(n+m_j).
\end{align*}
As in \secref{sec:notation}, we construct a second arithmetic function; here with the purpose of locating all $n$ such that $n+m_i$, $1\leq i \leq j$, is coprime to $p_j\#$:
\begin{align*}
	  \hat{R}_{j}(n) :=   \prod_{1\leq i \leq j} \hat{\rho}_{i}(n).
\end{align*}
The random variables $\tilde \Pi_j$ are now obtained by
\begin{align*}
	\tilde \Pi_j := |\{n: n\in s_j, \hat{R}_j(n)=1 \}|.
\end{align*}
Since the sequences $\hat{\rho}_j$, $1 \leq j \leq k$, are drawn independent from each other, the expectation value of $\tilde \Pi_j$ is given by
\begin{align*}
	\E\left[\tilde \Pi_j\right] = \tilde \pi_j = l_j \cdot \prod_{p\in \mathcal{P}_j} \left(1-\frac{1}{p} \right) 
				\sim 2 \e^{-\gamma} \frac{l_j}{\log p_{j+1}^2}
				\sim 2 \e^{-\gamma} \li_j,
\end{align*}
and correspondingly,
\begin{align*}
	\E\left[\tilde \Pi(p_{k+1}^2)\right] = \sum_{j=1}^k \tilde \pi_j \sim 2 \e^{-\gamma} \sum_{j=1}^k \li_j = 2 \e^{-\gamma}  \li(p_{i+1}^2).
\end{align*}
	
Note that the model as stated does not produce the correct mean. This can easily be overcome be introducing the random variables 
\begin{align*}
	\Pi_j := \frac{\e^{\gamma}}{2} \tilde \Pi_j \quad \textrm{and} \quad \Pi(p_{k+1}^2) := \sum_{j=1}^k \Pi_j.
\end{align*}
However, our interest is not in the mean, but in the variance produced by this model, and for that purpose it is not essential whether the mean differs by a constant; we still expect the primes to be constrained by the same variance. 

To continue, we write the variance of $\tilde \Pi(p_{k+1}^2)$ as
\begin{align*}
	\Var \left[\tilde \Pi(p_{k+1}^2) \right] 
	&= \E\left[ \left( \sum_{j=1}^k \left(\tilde \Pi_j - \tilde \pi_j \right) \right)^2\right]\\
	&= \sum_{j=1}^k \E\left[\left(\tilde \Pi_j - \tilde \pi_j \right)^2 \right] 
	+ 2 \sum_{1\leq i<j \leq k} \E\left[\left(\tilde \Pi_i - \tilde \pi_i \right)\cdot \left(\tilde \Pi_j - \tilde \pi_j \right) \right].
\end{align*}
As we see, the first expression after the last equality is a sum over the individual variances, while the second expression contains all covariance terms; they are all appearing twice, thereby explaining the factor 2. Crucially, we can express the variance theoretically, and will do so in the next section. 

For comparison, we also consider what we will refer to as the {\it uncorrelated random model}, where the sequences $\hat{\rho}_j$, $1\leq j \leq k$, are drawn anew for each interval $s_k$. We formulate this model in terms of the random variables
\begin{align*}
	\hat{\Pi}_j, \, 1\leq j \leq k, \quad \textrm{and} \quad \hat{\Pi}(p_{k+1}^2) := \sum_{j=1}^k \hat{\Pi}_j.
\end{align*}
The expectation values of these random variables are the same as above for the correlated random model, but now the covariances between intervals are zero, so that the variance is simply 
\begin{align*}
	\Var \left[\hat{\Pi}(p_{k+1}^2) \right] 
	= \sum_{j=1}^k \E\left[\left(\hat{\Pi}_j - \tilde \pi_j \right)^2 \right].
\end{align*}

To illustrate the behaviour of the random models, we generate 275 realisations of $\tilde \Pi(p_{k+1}^2)$, as well as 275 realisations of $\hat{\Pi}(p_{k+1}^2)$, for $1 \leq k \leq 10^4$. The result is displayed in \figref{fig:randommodels}, where the different realisations are plotted with their mean subtracted. Notable is that the correlated random model exhibits less variance than the uncorrelated model, indicating that the covariance terms in the correlated model in general are negative. In the same figure, we also plot the error function 
\begin{align*}
\epsilon(p_{k+1}^2) := \pi(p_{k+1}^2) - \li(p_{k+1}^2)
\end{align*}
as well as the adjusted error function  
\begin{align*}
\epsilon_0(p_{k+1}^2) := \pi(p_{k+1}^2) - 	\left[c\cdot \li\left(p_{k+1}^2\right)-(1-c)  \cdot  \sum_{j=1}^k \frac{l_j}{\log {p_{j+1}^2} }\right],
\end{align*}
where $c\approx 1-0.604$. It is evident that both of these error functions display a behaviour that is closer to that of the correlated model.

It is worth contrasting the correlated random model with the famous Cramer random model \citep{Cramer:1936ud} (see also \citep{Granville:1995} for a thorough review of Cramer's work). In Cramer's model each integer $n$, $1\leq n \leq x$, has the probability $1/ \log n$ of being prime, which gives a correct estimate of the number of primes across long enough intervals. While the model has proved very valuable in conjecturing different properties of the prime numbers, for example on the distribution of prime gaps, it also has the apparent weakness that it does not preserve certain central characteristics of the primes, such as the fact that an even number other than 2 can never be prime. In our suggested model, on the other hand, we move away from looking at each integer. Rather, the number of primes in each interval $s_k$ is treated as a random variable, and the correlation structure between intervals emerging from the underlying periodicities of $\rho_k$ is kept intact. In this sense, we expect our choice of model to better describe the behaviour of the prime counting function $\pi(x)$.

\begin{figure}
\centering
\includegraphics[width=8.7cm]{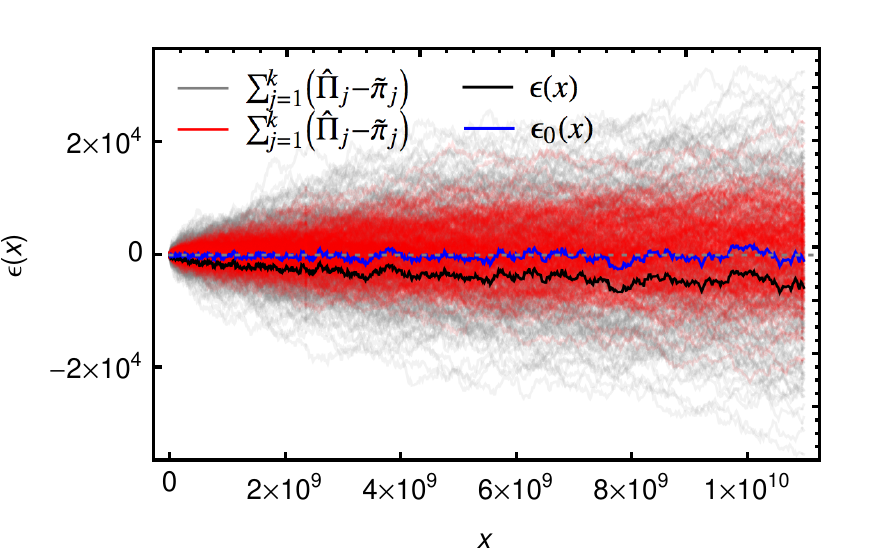}
\caption{
\small \sl
The error function $\epsilon(x)$ (black) and the adjusted error function $\epsilon_0(x)$ (blue) plotted at the values $x = p_{k+1}^2$, $1\leq k \leq 10^4$, together with 275 realisations of the correlated random model (red) and 275 realisations of the uncorrelated random model (grey). 
}
\label{fig:randommodels}
\end{figure}

\section{Theoretical expressions for the variance of $\tilde \Pi\left(p_{k+1}^2\right)$}
\label{sec:variance}

In an elementary calculation, Hausman and Shapiro \citep{hausman:1973} derived the variance of the number of reduced residues modulo $n$ in an arbitrary interval of length $h$. We here extend to the general case of the covariance between two intervals $A_1$ and $A_2$ with corresponding lengths $h_1$ and $h_2$ that are separated by a distance $q$, and where in $A_1$ we regard the number of reduced residues modulo $n_1$, and in $A_2$ the number of reduced residues modulo $n_2$. The resulting expressions for the covariance we present here follows from a straight forward adaption of Hausman and Shapiro's proof.  

Let 
\begin{align*}
    f_n(m):= 
\begin{cases}
    1 \quad \text{if} \quad (m,n) = 1,\\
    0 \quad \text{if} \quad (m,n) > 1.
\end{cases}
\end{align*}
The function of $f_n(m)$ is to weed out coprimes to $n$. To count the number of such in an interval of length $h$, we define
\begin{align*}
F_n(m,h) := \sum_{r=m}^{m+h-1} f_n(r).
\end{align*}
From a probabilistic perspective, the expected number of coprimes to $n$ in an interval of length $h$ is given by 
\begin{align*}
h \frac{\Phi (n)}{n},
\end{align*}
where $\Phi (n)$ is Euler's totient function. 

Denote now the number of coprimes to $n_1$ in $A_1$ and the number of coprimes to $n_2$ in $A_2$ by 
$F_{n_1}(m,h_1)$ and $F_{n_2}(m+q,h_2)$, respectively, where we assume $n_1\leq n_2$ and $q\geq 1$. Then the covariance between $F_{n_1}(m,h_1)$ and $F_{n_2}(m+q,h_2)$ is given by
\begin{align}
&G(n_1,n_2, h_1, h_2, q) \nonumber \\ 
	&{\qquad}:= \frac{1}{n_1 n_2}\sum_{m=1}^{n_1 n_2} \left( F_{n_1}(m, h_1) - h_1 \frac{\Phi ({n_1})}{n_1} \right) \left( F_{n_2}(m+q,h_2) - h_2 \frac{\Phi ({n_2})}{n_2} \right) \label{eq:G}\\
	&{\qquad}= \frac{1}{n_1 n_2}\sum_{m=1}^{n_1 n_2} F_{n_1}(m, h_1)F_{n_2}(m+q,h_2) - h_1 h_2 \frac{\Phi (n_1 n_2)}{n_1 n_2}.\nonumber 
\end{align}
Note that $G$ is periodic with respect to $q$ and has periodicity $n_1$,
\begin{align*}
G(n_1,n_2, h_1, h_2, q+n_1) = G(n_1,n_2, h_1, h_2, q).
\end{align*}

For brevity, we only continue with the case when $h_2>h_1$ and $n_1$ and $n_2$ are both even. Similar expressions can be derived in the other cases. Under this assumption, and introducing the abbreviation
\begin{align*}
	Q(n_1,n_2) :=
	\frac{1}{2}n_1 n_2	 
	\prod_{ \substack{ {p \mid \left( n_1 n_2/(n_1,n_2)^2 \right)} } } \left( 1-\frac{1}{p} \right) 	
	\prod_{ \substack{ {p \mid  (n_1,n_2)'} } } \left( 1-\frac{2}{p} \right),
\end{align*}
where $n'$ denotes the largest odd divisor of $n$, we can express the sum after the last equality in \eqref{eq:G} as
\begin{align*}
&\sum_{m=1}^{n_1 n_2} F_{n_1}(m, h_1)F_{n_2}(m+q,h_2)=\\
& \qquad 
	Q(n_1,n_2)
	\sum_{ \substack{ {k=1} \\  {2 \mid q+k-1} } }^{h_2-h_1+1} h_1  	
	\prod_{ \substack{ {p \mid (n_1,n_2)'} \\ {p \mid (q+k-1)} } } \left( 1+\frac{1}{p-2} \right)+ \\
& \qquad
	Q(n_1,n_2)
	\sum_{ \substack{ {k=1} \\  \\ {2 \mid (q+h_2-h_1+k)} } }^{h_1-1} (h_1 - k)
	\prod_{ \substack{ {p \mid (n_1,n_2)'} \\ {p \mid (q+h_2-h_1+k)} } } \left( 1+\frac{1}{p-2} \right)  + \\
& \qquad
	Q(n_1,n_2)
	\sum_{ \substack{ {k=1} \\  \\ {2 \mid (q-k)} } }^{h_1-1} (h_1 - k)
	\prod_{ \substack{ {p \mid (n_1,n_2)'} \\ {p \mid (q-k)} } } \left( 1+\frac{1}{p-2} \right). 
\end{align*}
For computational purposes, $G(n_1,n_2, h_1, h_2, q)$ is most efficiently calculated using this expression, and it underlies later numerical examples. But it is possible to derive a more compact theoretical expression in terms of the M\"obius function [see \eqref{eq:mobius}].
We also need the following definitions:
For any square-free integer $n$, 
\begin{align*}
\rho(n) := \prod_{p \mid n}(p-2);
\end{align*}

\begin{align*}
\gamma^+(h,q,d) := 
\begin{cases}
0 \quad \textrm{if } d \nmid q+i \textrm{ for all $i$, where } 1 \leq i \leq d \left\{ \frac{h}{d} \right\}, \\
1 \quad \textrm{if } d \mid q+i \textrm{ for some $i$, where } 1 \leq i \leq d \left\{ \frac{h}{d} \right\};
\end{cases}
\end{align*}
\begin{align*}
\gamma^-(h,q,d) := 
\begin{cases}
0 \quad \textrm{if } d \nmid q-i \textrm{ for all $i$, where } 1 \leq i \leq d \left\{ \frac{h}{d} \right\}, \\
1 \quad \textrm{if } d \mid q-i \textrm{ for some $i$, where } 1 \leq i \leq d \left\{ \frac{h}{d} \right\};
\end{cases}
\end{align*}
and
\begin{align*}
\left\{ \frac{q}{d} \right\}_1 := 
\begin{cases}
\left\{ \frac{q}{d} \right\} \quad \textrm{if } d \nmid q,\\
1 \quad \textrm{if } d \mid q.
\end{cases}
\end{align*}

Then we can write the covariance $G(n_1,n_2, h_1, h_2, q)$ on the equivalent form 
\begin{align}
	G(n_1,n_2, h_1, h_2, q) 
	= 
	\frac{1}{n_1 n_2}Q(n_1,n_2) \sum_{ \substack{ {d \mid (n_1,n_2)'} } } \frac{\mu^2(d)}{\rho(d)} \B(h_1,h_2,q,d),
\label{eq:Gcompact}	 
\end{align}
where 
\begin{align*}
	&
	\B(h_1,h_2,q,d)\\
	&{\qquad}=
	h_1 
	\left(
	\frac{1}{2d} - \left\{ \frac{h_2-h_1+1}{2d} \right\}  +  \left\{ \frac{q+h_2-h_1}{2d} \right\} - \left\{ \frac{q}{2d} \right\}_1
	\right)\\
	&{\qquad}-
	2 d 
	\left\{ \frac{h_1}{2d} \right\}
	\left(
	\left\{ \frac{h_1}{2d} \right\} + \left\{ \frac{q+h_2-h_1}{2d} \right\} - \left\{ \frac{q}{2d} \right\}_1
	\right)\\
	&{\qquad}
	+ \gamma^+(h_2-h_1+1,q-1,2d) h_1\\
	&{\qquad}
	+\gamma^+(h_1,q+h_2-h_1,2d) 2d \left( \left\{ \frac{h_1}{2d} \right\} -\left(1-\left\{ \frac{q+h_2-h_1}{2d} \right\} \right) \right)
	\\
	&{\qquad}
	+\gamma^-(h_1,q,2d) 2d \left( \left\{ \frac{h_1}{2d} \right\} - \left\{ \frac{q}{2d} \right\}_1 \right). \\
\end{align*}
Numerical investigations hints to a possible simplification of this expression, as some of the terms appear to always cancel each other, but this remains to be shown theoretically. 

Whenever $A_1=A_2$, so that $h_1=h_2=h$ and $n_1=n_2=n$, we can write $H(n, h):=G(n, n, h, h, 0)$, and \eqref{eq:Gcompact} reduces to the variance derived in \citep{hausman:1973},
\begin{align*}
H(n,h) 
	= \prod_{ \substack{ {p \mid n'} } } \left( 1-\frac{2}{p} \right)	 
	\sum_{ \substack{ {d \mid n'} } }  \frac{\mu^2(d)}{\rho(d)}		
	d \left\{\frac{h}{2d}\right\} \left( 1 - \left\{\frac{h}{2d}\right\} \right).	
\end{align*}
Note that if we apply the inequality $\{x\} ( 1 - \{x\} ) \leq x$, we obtain the upper bound 	
\begin{align}
H(n,h) \leq h \frac{\Phi(n)}{n}.
\label{eq:variancebound}
\end{align}

Returning to the random models in \secref{sec:randommodel}, we can now express the variance of the correlated model as
\begin{align*}
	\Var \left[\tilde \Pi(p_{k+1}^2) \right] 
	&= \sum_{j=1}^k \E\left[\left(\tilde \Pi_j - \tilde \pi_j \right)^2 \right] 
	+ 2 \sum_{1\leq i<j \leq k} \E\left[\left(\tilde \Pi_i - \tilde \pi_i \right)\cdot \left(\tilde \Pi_j - \tilde \pi_j \right) \right]\\
	&= \sum_{j=1}^k H(p_j\#,l_j)
	+ 2 \sum_{1\leq i<j \leq k} G(p_i\#,p_j\#, l_i, l_j, p_{j+1}^2-p_{i+1}^2).
\end{align*}
Likewise, for the uncorrelated model,
\begin{align*}
	\Var \left[\hat{\Pi}(p_{k+1}^2) \right] 
	= \sum_{j=1}^k \E\left[\left(\hat{\Pi}_j - \tilde \pi_j \right)^2 \right] 
	= \sum_{j=1}^k H(p_j\#,l_j).
\end{align*}
From \eqref{eq:variancebound} it follows that 
\begin{align*}
	H(p_j\#,l_j) \leq l_j \cdot \frac{\Phi(p_j\#)}{p_j\#} 
	= \l_j \cdot \prod_{p\in \mathcal{P}_j} \left(1-\frac{1}{p} \right) 
	\sim 2 \e^{-\gamma} \frac{l_j}{\log p_{j+1}^2} 
	\sim 2 \e^{-\gamma} \li_j,
\end{align*}
which results in the variance of the uncorrelated model having the upper bound 
\begin{align*}
	\Var \left[\hat{\Pi}(p_{k+1}^2) \right] 
	\leq 2 \e^{-\gamma} \sum_{j=1}^k \li_j = 2 \e^{-\gamma} \li \left (p_{k+1}^2 \right).
\end{align*}

\section{Properties of the covariance function}
\label{sec:historicalnote}
In this section we present some general considerations on the behaviour of the covariance function $G(n_1,n_2, h_1, h_2, q)$.  Let us first have a look at the role of the relative position $q$ of two intervals, when both are assumed to be sieved by the $k$ first primes. As an example of this scenario, we show in \figref{fig:negcorr}A a plot of $G(p_{100}\#,p_{100}\#, h_1, 2 h_1, q)$ as a function of $q$ for different values of $h_1$. We observe that in each case, the covariance function reaches a negative minimum value when $q=h_1$---the exact situation when the two intervals are adjacent---before slowly climbing towards zero again. In addition, the minimum values grow in size as the length of the intervals increases. While not shown in \figref{fig:negcorr}A, this result also holds for arbitrary values of $h_2$. Note, however, that the observed minimum values are not unique, as shown in \figref{fig:negcorr}B, where  we have plotted $G(p_k\#,p_k\#, h_1, h_1, q)$ across its full period $p_k\#$ for two values of $h_1$.

The reason why $q=h_1$ corresponds to a minimum value of the covariance function can be understood in terms of the following heuristic:  For the sake of argument, assume that $A_1$ and $A_2$ are both sieved by $p_k$ only, and that their lengths $h_1$ and $h_2$ satisfy $h_1 + h_2 \leq p_k$. Then we have that 
\begin{align*}
F_{p_k}(m,h_1) = h_1-1 \quad \textrm{or} \quad F_{p_k}(m,h_1) = h_1,
\end{align*}
depending respectively on whether the interval $A_1$ lies across a multiple of $p_k$ or not. Likewise in the case of $A_2$:
\begin{align*}
F_{p_k}(m,h_2) = h_2-1, \quad \textrm{or} \quad F_{p_k}(m,h_2) = h_2.
\end{align*}
Under these assumptions, we have that 
\begin{align*}
G(p_k,p_k, h_1, h_2, q) = \frac{1}{p_k}\sum_{m=1}^{p_k} F_{p_k}(m, h_1)F_{p_k}(m+q,h_2) - h_1 h_2 \left( \frac{p_k-1}{p_k} \right)^2.
\end{align*}
When $A_1$ and  $A_1$ do not overlap, that is, $q \geq h_1$, we can write the sum after the equality in terms of three contributions,
\begin{align*}
&\sum_{m=1}^{p_k} F_{p_k}(m, h_1)F_{p_k}(m+q,h_2)\\
&\qquad = (p_k - h1-h_2)\cdot[h_1 h_2] + h_1\cdot[h_2 (h_1-1)] + h_2 \cdot[h_1 (h_2-1)]\\
&\qquad = h_1 h_2 (p_k-2).
\end{align*}
And if $A_1$ and $A_2$ do overlap, that is, $0\leq q < h_1$, in terms of four contributions,
\begin{align*}
&\sum_{m=1}^{p_k} F_{p_k}(m, h_1)F_{p_k}(m+q,h_2)\\
&\qquad = (p_k -h_2-q)\cdot[h_1 h_2] + q\cdot [(h_1-1)h_2 ] + (h_1-q)\cdot[(h_1-1)(h_2-1)] \\
&\qquad \qquad+ (h_2-h_1+q)\cdot[h_1(h_2-1)]\\
&\qquad = -q+h_1 +h_1 h_2 (p_k-2).
\end{align*}
Comparing these two sums, we see that $G(p_k,p_k, h_1, h_2, q)$  is minimised whenever $q\geq h_1$, in which case
\begin{align*}
G(p_k,p_k, h_1, h_2, q) = -\frac{h_1 h_2}{p_k^2} <0.
\end{align*}
Note that the larger the product $h_1 h_2$ is, the larger the absolute value of the covariance at its minimum, as we observe in \figref{fig:negcorr}A. 

Letting go of the restriction that $h_1+h_2 \leq p_k$ complicates the heuristic, since more careful accounting is needed; as does sieving the intervals by multiple primes. It seems realistic, however, that a proof can be produced to show that in general, $G(p_k\#,p_k\#, h_1, h_2, q)$ takes on a negative minimum value whenever $q=h_1$.


\begin{figure}
\centering
\includegraphics[width=360pt]{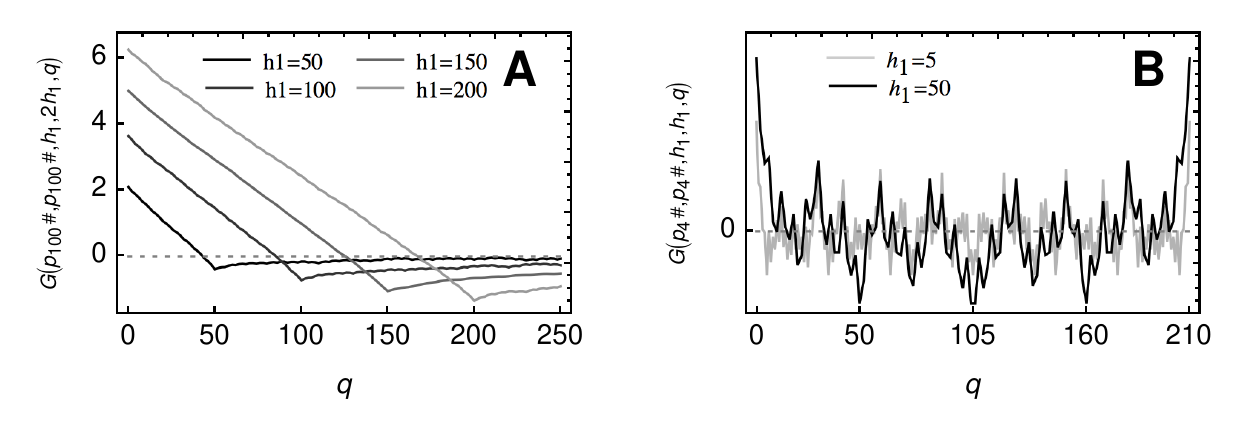}
\caption{
\small \sl
Covariance plots: (A) $G(p_{100}\#, p_{100}\#, h_1, 2 h_1, q)$ plotted for four different values of $h_1$ as a function of $q$. In each case, the covariance function is minimum at $q=h_1$. (B) $G(p_4\#, p_4\#, h_1, h_1, q)$ plotted for two different values of $h_1$ as a function of $q$ across its full period. The plot illustrates that the minimum found at $q=h_1$ is in general not unique. 
}
\label{fig:negcorr}
\end{figure}

Aside the relative positions of intervals, a second aspect to consider is what primes each interval is sieved by. Again, assume two intervals, now with lengths $h_1$ and $h_2$, where the first is sieved by the $i$ first primes, and the second by the $j$ first primes. Then we can write the covariance function as
\begin{align*}
G(p_i\#,p_j\#, h_1, h_2, q) 
= c_1 \prod_{ \substack{ {p \mid \left( p_j\#/p_i\# \right) } } } \left( 1-\frac{1}{p} \right)
- c_2 \prod_{ \substack{ {p \in \mathcal{P}_j } } } \left( 1-\frac{1}{p} \right),
\end{align*}
where $c_1$ and $c_2$ are constants. Both products approaches 0 as $j\rightarrow \infty$, and therefore we have that 
\begin{align*}
\lim_{j\rightarrow \infty} G(p_i\#,p_j\#, h_1, h_2, q) = 0.
\end{align*}
This limit behaviour is illustrated in \figref{fig:limitsieve}; $G(p_i\#,p_j\#, h_1, h_2, q)$ is plotted for two values of $q$, and in both cases it moves slowly towards 0 as $j$ increases. 

Combining these two insights---the covariance function is minimised for adjacent intervals and approaches zero as the difference in sieving steps between intervals increases--- 
we anticipate that $G(p_i\#,p_j\#, l_i, l_j, p_{j+1}^2-p_{i+1}^2)$
is negative when $j-i$ is small, increasingly so when $i$ and $j$ both become large, and grows smaller in absolute value as the difference $j-i$ increases. Hence, the sum  
\begin{align*}
	\sum_{1\leq i<j \leq k} G(p_i\#,p_j\#, l_i, l_j, p_{j+1}^2-p_{i+1}^2)
\end{align*}
should be dominated by negative terms, implying that the variance of the correlated random model is bounded by the same upper bound as the uncorrelated random model. Thus,
\begin{align*}
	\Var \left[\tilde \Pi(p_{k+1}^2) \right] 	\leq \sum_{i=1}^k H(p_i\#,l_i) \leq 2 \e^{-\gamma}\li(p_{k+1}^2).
\end{align*}

\begin{figure}
\centering
\includegraphics[width=250pt]{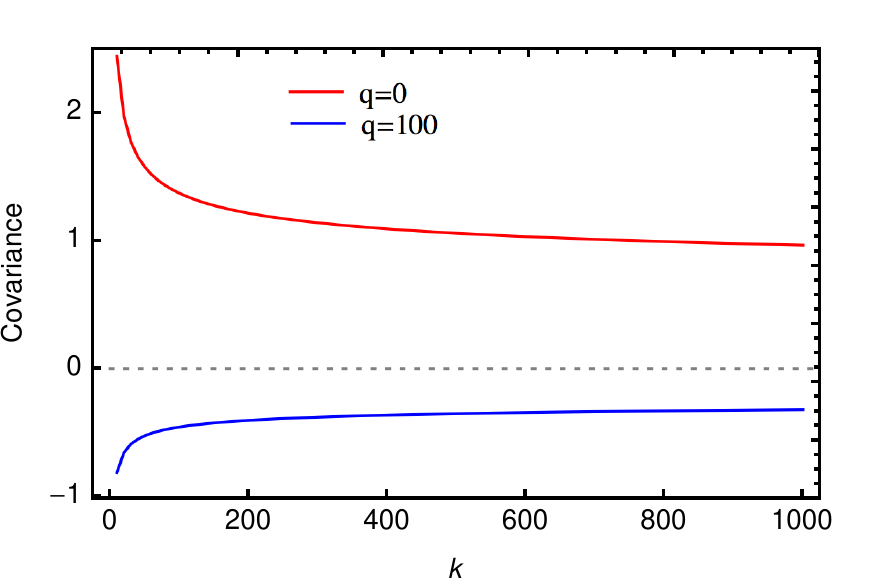}
\caption{
\small \sl
The covariance function $G(p_{10}\#, p_{k}\#, 100, 100, q)$ plotted for two different values of $q$ as a function of $10 \leq k \leq 1000$. 
}
\label{fig:limitsieve}
\end{figure}

\section{Comparing random models and empirical data}
\label{sec:historicalnote}

In order to verify the predictions stated in the previous section we now consider four cases: 1) The theoretical variance of the correlated random model; 2) The variance of the correlated random model estimated from sampled realisations of $\tilde \Pi\left(p_{k+1}^2\right)$; 3) The estimated variance of the prime counting function $\pi(x)$ under the assumption that $E[\pi(x)]=\li(x)$; and 4) The estimated variance of the prime counting function $\pi(x)$ under the assumption that 
\begin{align*}
	E[\pi\left(p_{k+1}^2\right)]=c\cdot \li\left(p_{k+1}^2\right)-(1-c)  \cdot  \sum_{i=1}^k l_i/\log {p_{i+1}^2}.
\end{align*}

Let us start with the separate contribution from the covariance terms between intervals. In the theoretical case, we write the sum  
\begin{align*}
	\sum_{1\leq i <j \leq k} G(p_i\#,p_j\#, l_i, l_j, p_{j+1}^2-p_{i+1}^2)
\end{align*}
on the equivalent form
\begin{align*}
	\sum_{j=1}^{k-1} \left( \sum_{i=1}^{k-j} G(p_i\#,p_{i+j}\#, l_i, l_{i+j}, p_{i+j+1}^2-p_{i+1}^2) \right),
\end{align*}
and denote the expression in parentheses by 
\begin{align*}
	\kappa_{\textrm{th}}(j) := \sum_{i=1}^{k-j} G(p_i\#,p_{i+j}\#, l_i, l_{i+j}, p_{i+j+1}^2-p_{i+1}^2),
\end{align*}
where the label "th" is short for theory. $\kappa_{\textrm{th}}(j)$ is therefore the sum over all covariance terms originating from $j$th nearest intervals. For example, $\kappa_{\textrm{th}}(1)$ gives the sum over all covariance terms for neighbouring intervals. We also need the total remaining contribution to the covariance for all $j$ larger than some value $d$, which we obtain by
\begin{align*}
	\sum_{j>d}^{k-1} \kappa_{\textrm{th}}(j).
\end{align*}

Now to the second case, where we draw samples of $\tilde \Pi\left( p_{k+1}^2 \right)=\sum_{i=1}^k \tilde \Pi_k$.  We define the error function in each interval $j$ as $\epsilon_{\textrm{sim},j}:=\tilde \Pi_j - \tilde \pi_j$, 
where the label "sim" is short for simulation, and write the global error function as 
\begin{align*}
	\epsilon_{\textrm{sim}}(p_{k+1}^2) := \sum_{j=1}^k (\tilde \Pi_j - \tilde \pi_j) = \sum_{j=1}^k \epsilon_{\textrm{sim},j}.
\end{align*}
Then the measured variance of one realisation of $\tilde \Pi\left( p_{k+1}^2 \right)$ can be expressed as 
\begin{align*}
	\Var\left[\tilde \Pi \left( p_{k+1}^2 \right) \right] 
	= \left(\sum_{j=1}^k \epsilon_{\textrm{sim},j}\right)^2 
	= \sum_{j=1}^k \epsilon_{\textrm{sim},j}^2 + 2 \sum_{1\leq i<j \leq k} \epsilon_{\textrm{sim},i} \cdot \epsilon_{\textrm{sim},j}.
\end{align*}

Similar to the theoretical case, we rewrite the sum  
\begin{align*}
	\sum_{1\leq i<j \leq k} \epsilon_{\textrm{sim},i} \cdot \epsilon_{\textrm{sim},j}
\end{align*}
as
\begin{align*}
	\sum_{j=1}^{k-1} \left(  \sum_{i=1}^{k-j} \epsilon_{\textrm{sim},i} \cdot \epsilon_{\textrm{sim},i+j} \right).
\end{align*}
Since we are interested in the average taken over many realisations, we define
\begin{align*}
	\kappa_{\textrm{sim}}(j) := \left\langle \sum_{i=1}^{k-j} \epsilon_{\textrm{sim},i} \cdot \epsilon_{\textrm{sim},i+j} \right\rangle,
\end{align*}
and the total remaining contribution to the covariance for all $j$ larger than some value $d$, by
\begin{align*}
	\sum_{j>d}^k \kappa_{\textrm{sim}}(j).
\end{align*}

Finally, in the cases of the primes, we define 
\begin{align*}
	\epsilon_{pr,j}:=\pi_j - \li_j \quad \textrm{and} \quad
	\epsilon_{ad,j}:=\pi_j -\left[c\cdot \li_j - (1-c)  \cdot \frac{l_j}{\log {p_{j+1}^2} }\right],
\end{align*}
where "pr" and "ad" denote primes and primes with adjusted mean, and obtain
\begin{align*}
	\epsilon_{pr}\left(p_{k+1}^2\right) :=  \sum_{j=1}^k \epsilon_{pr,j} 
	\quad \textrm{and} \quad
	\epsilon_{ad}\left(p_{k+1}^2\right) 
	:= \sum_{j=1}^k \epsilon_{ad,j}.
\end{align*}
Then, entirely analogous to the previous cases, we obtain 
\begin{align*}
	\kappa_{\textrm{pr}}(j) := \sum_{i=1}^{k-j} \epsilon_{pr,i} \cdot \epsilon_{pr,i+j} \quad \textrm{and} \quad
	\sum_{j>d}^k \kappa_{\textrm{pr}}(j),
\end{align*}
and
\begin{align*}
	\kappa_{\textrm{ad}}(j) := \sum_{i=1}^{k-j} \epsilon_{ad,i} \cdot \epsilon_{ad,i+j} \quad \textrm{and} \quad
	\sum_{j>d}^k \kappa_{\textrm{ad}}(j).
\end{align*}

We now plot $\kappa_{\textrm{th}}(j)$, $\kappa_{\textrm{sim}}(j)$, $\kappa_{\textrm{pr}}(j)$, and $\kappa_{\textrm{ad}}(j)$ for different values of $j$ up to a maximum $d$, as well as the remainder terms, all shown in \figref{fig:covarianceSums}. The different plots are across different ranges of $k$---limited to smaller $k$ in the first two cases due to increasing computational cost for large $k$---but we observe that they are all qualitatively similar, confirming that the primes behave according to the correlated random model. For direct comparison, we include $\kappa_{\textrm{pr}}(1)$ and $\kappa_{\textrm{pr}}(2)$ also in the plots of $\kappa_{\textrm{th}}(j)$ and $\kappa_{\textrm{sim}}(j)$. At this short range, however, the fluctuations of $\kappa_{\textrm{pr}}(1)$ and $\kappa_{\textrm{pr}}(2)$ are rather significant, so we are not able to determine whether $\kappa_{\textrm{pr}}(j)$ converges to the theoretical value of the correlated random model in the limit of large $k$. In fact, as will be evident when we plot the total variance in each case, it appears that the covariance functions $\kappa_{\textrm{pr}}(j)$ and $\kappa_{\textrm{ad}}(j)$ in general lie below $\kappa_{\textrm{th}}(j)$ as $k$ grows large, and this might be what we actually observe in \figref{fig:covarianceSums}A and more convincingly in \figref{fig:covarianceSums}B. 

As we anticipated in the previous section, the contribution to the covariance is largest when $j$ is small; $\kappa_{\textrm{a}}(1)$ is more than a factor two larger than $\kappa_{\textrm{a}}(2)$ (here "a" denotes any of "th", "sim",  "pr", or "ad"). And as $j$ grows larger, $\kappa_{\textrm{a}}(j)$ approaches zero in absolute value. Note though that while we can observe some $\kappa_{a}(j)$ being positive, the remainder terms in all cases add up to negative values. 

Comparing \figref{fig:covarianceSums}C and \figref{fig:covarianceSums}D, we see that for small $j$, $\kappa_{\textrm{pr}}(j)$ and $\kappa_{\textrm{ad}}(j)$ are almost identical. However, for $j$ large the remainder terms differ significantly. Assuming $\li(x)$ to be the correct expectation value of $\pi(x)$ produces large fluctuations in the remainder term, while a much smoother curve results from assuming the adjusted expectation value. 

\begin{figure}
        \centering
        \begin{subfigure}[b]{\textwidth}
	        \centering
                \includegraphics[width=360pt]{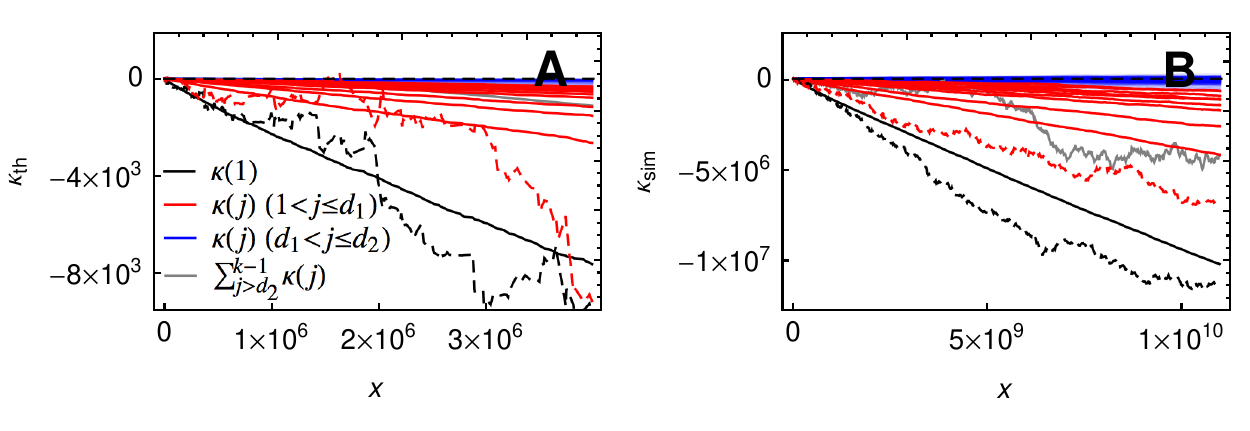}
        \end{subfigure}%
        ~ 

        \begin{subfigure}[b]{\textwidth}
        \centering
                \includegraphics[width=360pt]{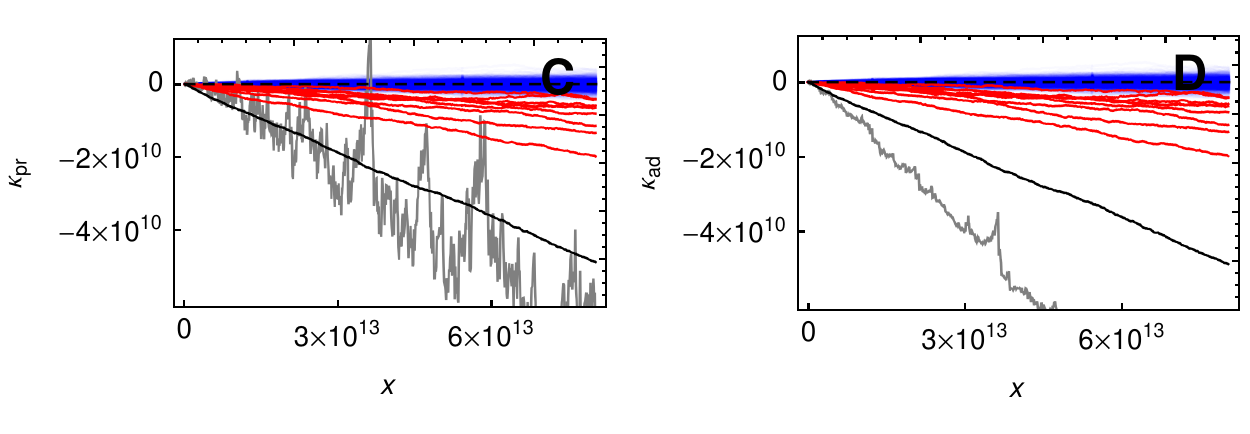}
        \end{subfigure}
\caption{
\small \sl
The covariance sums $\kappa_{\textrm{th}}(j)$ (A), $\kappa_{\textrm{sim}}(j)$ (B), $\kappa_{\textrm{pr}}(j)$ (C), and $\kappa_{\textrm{ad}}(j)$ (D) plotted at the values $x = p_{k+1}^2$. The legend applies for all three covariance sums. In (A), $1\leq k \leq 300$, and data are plotted for each value of $k$, $d_1=10$, and $d_2=50$; In (B), $1\leq k \leq 10^3$,  and data are plotted for every 10th value of $k$, $d_1=10$, and $d_2=200$. The curves are obtained by averaging over 275 individual samples; In (C) and (D), $1\leq k \leq 6\times 10^5$, and data are plotted for every 500th value of $k$, $d_1=10$, and $d_2=1000$. In (A) and (B), $\kappa_{\textrm{pr}}(1)$ and $\kappa_{\textrm{pr}}(2)$ are included for comparison (dashed black and red, respectively).
}
\label{fig:covarianceSums}
\end{figure}

Next, we continue by examining the total variance in each case, or rather, the standard deviations, that are given by 
\begin{align*}
	\sigma_{\textrm{th}}\left( p_{k+1}^2 \right) 
	&:= \sqrt{\Var\left[\tilde \Pi(p_{k+1}^2) \right]}, \\
	\sigma_{\textrm{sim}}\left( p_{k+1}^2 \right) 
	&:= \sqrt{\left\langle \left(\sum_{i=1}^k \epsilon_{\textrm{sim},j}\right)^2  \right\rangle}, \\
	\sigma_{\textrm{pr}}\left( p_{k+1}^2 \right) 
	&:= \sqrt{ \left(\sum_{j=1}^k \epsilon_{\textrm{pr},j}\right)^2 }, \quad \textrm{and}\\
	\sigma_{\textrm{ad}}\left( p_{k+1}^2 \right) 
	&:= \sqrt{ \left(\sum_{j=1}^k \epsilon_{\textrm{ad},j}\right)^2 }.
\end{align*}
In addition, we consider the standard deviations under the assumption that all covariance terms are zero, that is, corresponding to the uncorrelated random model. Hence,
\begin{align*}
	\sigma_{\textrm{th},0}\left( p_{k+1}^2 \right) 
	&:= \sqrt{\sum_{j=1}^k H(p_j\#,l_j)},\\
	\sigma_{\textrm{sim},0}\left( p_{k+1}^2 \right) 
	&:= \sqrt{\left\langle \sum_{j=1}^k \epsilon_{\textrm{sim},j}^2\right\rangle}, \\
	\sigma_{\textrm{pr},0}\left( p_{k+1}^2 \right) 
	&:= \sqrt{\sum_{j=1}^k \epsilon_{\textrm{pr},j}^2}, \quad \textrm{and}\\
	\sigma_{\textrm{ad},0}\left( p_{k+1}^2 \right) 
	&:= \sqrt{\sum_{j=1}^k \epsilon_{\textrm{ad},j}^2}\, .
\end{align*}

Finally, we include the upper bound we found earlier for the standard deviation of the uncorrelated random model:
\begin{align*}
	\sigma_{\textrm{ub}}\left( p_{k+1}^2 \right) := \sqrt{2 \e^{-\gamma}\li\left( p_{k+1}^2 \right) }.
\end{align*}

The different standard deviations are plotted in \figref{fig:stddev}. Under the assumption of the uncorrelated model, we observe that the curves of $\sigma_{\textrm{pr},0}$ and $\sigma_{\textrm{ad},0}$ lie on top of each other and are not visibly different. 
They also lie close to the theoretical value $\sigma_{\textrm{th},0}$, as shown in \figref{fig:stddev}A, and in fact seems to converge to $\sigma_{\textrm{sim},0}$, as evidenced in \figref{fig:stddev}B.

Including correlations, we note in \figref{fig:stddev}A that $\sigma_{\textrm{pr}}$ starts out larger than $\sigma_{\textrm{th}}$, but seems to stabilise below $\sigma_{\textrm{sim}}$ in \figref{fig:stddev}B. It appears that this trends continues in \figref{fig:stddev}C, but without comparison with the correlated random model, we cannot state this with certainty. In all cases, $\sigma_{\textrm{ad}}$ is small compared to $\sigma_{\textrm{pr}}$. 

As expected from the covariance contribution being negative, the total variance is always smaller than the variance of the uncorrelated model, which again is bounded by $\sigma_{\textrm{ub}}$, placed far above any of the other curves. 

\begin{figure}
        \centering
        \begin{subfigure}[b]{\textwidth}
	        \centering
                \includegraphics[width=360pt]{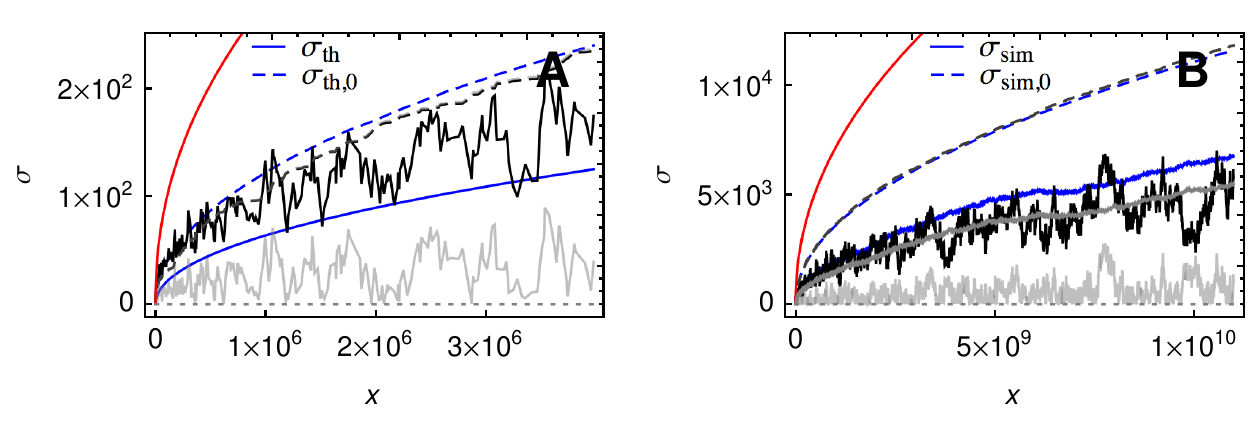}
        \end{subfigure}%
        ~ 
          
        \begin{subfigure}[b]{\textwidth}
        \centering
                \includegraphics[width=180pt]{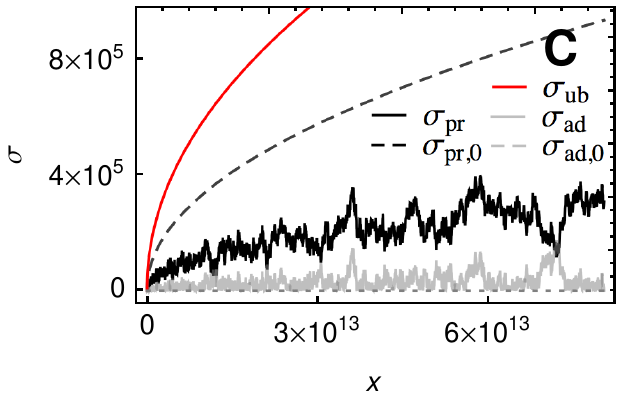}
        \end{subfigure}
\caption{
\small \sl
Different standard variance functions plotted at the values $x = p_{k+1}^2$ (calculated from theory, samples from the correlated random model, and from the prime counting function $\pi(x)$). The legends in (C) also applies to (A) and (B). In (A), $1\leq k \leq 300$, and data are plotted for each value of $k$; In (B), $1\leq k \leq 10^3$,  and data are plotted for every 10th value of $k$. The curves for $\sigma_{sim}$ and $\sigma_{sim,0}$ are obtained by averaging over 275 individual samples; In (C), $1\leq k \leq 6\times 10^5$, and data are plotted for every 100th value of $k$.         
}
\label{fig:stddev}
\end{figure}
\section{The distribution of primes and the Riemann hypothesis}
\label{sec:historicalnote}

Considered an outstanding challenge in number theory, the Riemann hypothesis states that the real part of every non-trivial zero of the Riemann zeta function is 1/2, where the Riemann zeta function is defined by
\begin{align*}
	\zeta(s) := \sum_{n=1}^{\infty} \frac{1}{n^s}. 
\end{align*}
The Riemann zeta function ties in tightly with the distribution of primes through the Euler product formula, which states that
\begin{align*}
	\zeta(s) = \prod_{p \in \mathcal{P}} \left(1-\frac{1}{p^s} \right)^{-1}.
\end{align*}
Our particular interest is in the fact that the Riemann hypothesis is equivalent to the upper bound on the error function being
\begin{align*}
	|\pi(x) - \li(x)| =O\left(\sqrt{x} \log x\right),	
\end{align*}
a statement proven by von Koch in 1901 \citep{vonKoch:1901ui}.

Koch's criteria is far from being a tight upper bound; it is much wider than what we have conjectured---and justified---in the previous sections, namely that for $x$ large enough,
\begin{align}
	|\pi(x) - \li(x)| < 2 \e^{-\gamma}\sqrt{\li(x)}.
\label{eq:bounduncorrelated}
\end{align}
In fact, we know enough to conjecture an even tighter upper bound, expressed discretely at the values $x=p_{k+1}^2$, as 
\begin{align}
	|\pi\left(p_{k+1}^2\right) - \li \left(p_{k+1}^2\right)| 
	<\sqrt{\sum_{j=1}^k \frac{ l_j \log \left( p_{j+1}^2/l_j \right) }{ \left( \log p_{j+1}^2 \right)^2 }}.
\label{eq:errorconjecturediscrete}
\end{align}
This statement follows from Montgomery and Soundararajan's conjecture discussed in \secref{sec:numresults}; the expression on the right side of the inequality is an estimate of the standard deviation of the uncorrelated random model. Due to the negative contribution to the variance from the covariance terms between intervals, this bound holds stance for the correlated random model and hence the primes. Furthermore,
since Montgomery and Soundararajan's conjecture holds empirically, we should expect that the bound should align with the sum of squares of the error functions in each interval, 
\begin{align}
	\sum_{j=1}^k (\pi_j - \li_j)^2.
\label{eq:sumofsquares}	
\end{align}

We can also express \eqref{eq:errorconjecturediscrete} in continuous form. Observe that under the assumption of Cramer's conjecture \citep{Cramer:1936ud}---that is, $g_j:=p_{j+1}-p_j =O\left( (\log p_j)^2\right)$---we have that 
\begin{align*}
	\log \left( p_{j+1}^2/l_i \right) 
	= \log p_{j+1}^2 - \log \left(2 g_j p_{j+1} - g_j^2 \right) 
	\sim \frac{1}{2} \log p_{j+1}^2,
\end{align*}
which further implies
\begin{align*}
\sum_{j=1}^k \frac{ l_j \log \left( p_{j+1}^2/l_j \right) }{ \left( \log p_{j+1}^2 \right)^2 }
\sim
\frac{1}{2} \sum_{j=1}^k \frac{ l_j }{ \log p_{j+1}^2  } 
\sim
\frac{1}{2} \li \left (p_{k+1}^2 \right).
\end{align*}
We can therefore restate \eqref{eq:errorconjecturediscrete} as
\begin{align}
	|\pi(x) - \li(x)| < \sqrt{\frac{1}{2}\li(x)}.
\label{eq:contestimate}
\end{align}
Neglecting any constant coefficients, we state this conjecture more generally as 
\begin{conjecture}
The error term in the prime number theorem satisfies
\begin{align*}
	|\pi(x) - \li(x)| = O\left(\sqrt{\li(x)}\right).
\end{align*}
\end{conjecture}

We continue by plotting the error function $|\pi(x) - \li(x)|$ together with the different upper bounds, as shown in \figref{fig:riemann}. Note that because Koch's criteria is much larger than the other bounds, we use a logarithmic scale on the $y$-axis. As expected, we find that the discrete estimate in \eqref{eq:errorconjecturediscrete}, corresponding to the uncorrelated random model, lies on top of the sum of squares \eqref{eq:sumofsquares} (in the plot we include the lower order term for higher accuracy), while the error function $|\pi(x) - \li(x)|$ is placed well below these bounds. The continuous estimate \eqref{eq:contestimate} lies close to the discrete estimate \eqref{eq:errorconjecturediscrete}, but a bit higher, since the lower order terms are not accounted for here; nonetheless, as $x\rightarrow \infty$, we expect their relative distance to decrease. Slightly above these curves we find the upper bound of the uncorrelated random model \eqref{eq:bounduncorrelated}. And finally, we observe that Koch's criteria for the Riemann hypothesis is residing high above all the other curves, leaving plenty of room below it. 

Essentially then, by considering the prime counting function $\pi(x)$ as a sum over correlated random variables, we have arrived at a fundamental explanation of why the Riemann hypothesis must be correct, as illustrated by the different theoretical bounds displayed in \figref{fig:riemann}. A purely technical proof still awaits, but perhaps what we have presented here can serve as a stepping stone towards one.

\begin{figure}
\centering
\includegraphics[width=250pt]{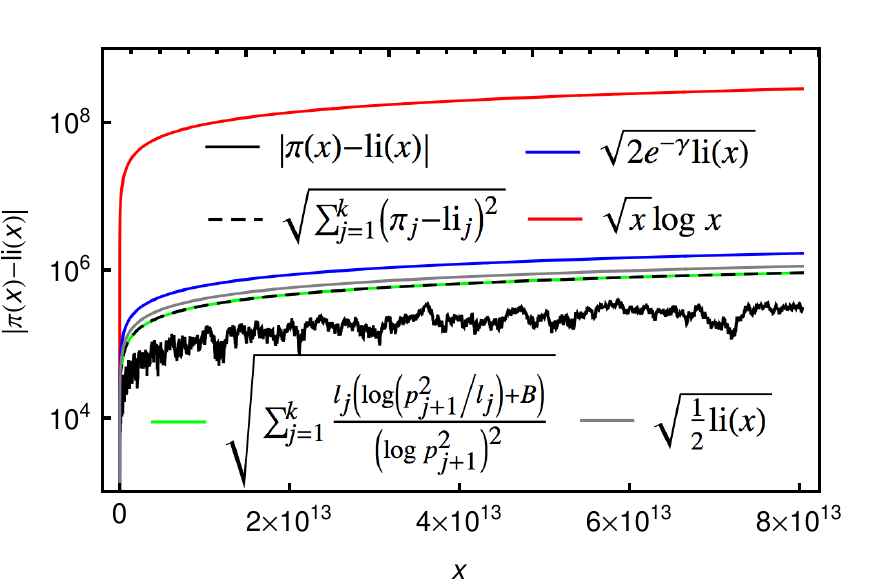}
\caption{
\small \sl
The error function $|\pi(x) - \li(x)|$ and different upper bounds plotted at the values $x = p_{k+1}^2$  for every 100th value of $k$, $1\leq k \leq 6\times 10^5$. 
%
%
%
}
\label{fig:riemann}
\end{figure}

\small
\bibliography{primereferences}
\vspace{2mm}
\bibliographystyle{plainnat}

\end{document}